\documentclass{amsart}

\usepackage{amssymb, amsmath, esint, xcolor}
\usepackage[backref=page]{hyperref}

\usepackage{mathtools}
\mathtoolsset{showonlyrefs}

\renewcommand{\[}{\begin{equation}\begin{aligned}}
\renewcommand{\]}{\end{aligned} \end{equation}}

\newcommand{\dist}{\overline{\mathrm{dist}}} 

\usepackage{verbatim}

\newtheorem{thm}{Theorem}
\newtheorem{prop}[thm]{Proposition}
\newtheorem{lemma}[thm]{Lemma}

\theoremstyle{remark}
\newtheorem{remark}[thm]{Remark}

\theoremstyle{definition}

\author{G\'abor Sz\'ekelyhidi}
\thanks{Supported in part by NSF grant DMS-2506325.}
\address{Department of Mathematics, Northwestern University, Evanston,
  IL, USA}
\email{gaborsz@northwestern.edu}

\title[Nondegenerate neck pinches]{Nondegenerate neck pinches along
  the mean curvature flow}
\date{}

\begin{document}

\begin{abstract}
  We show that for generic smooth compact initial surfaces the mean curvature flow
  in $\mathbb{R}^3$ has spherical or
  nondegenerate neck pinch singularities at the first singular
  time. In particular the singularities at the first singular time are
  isolated in spacetime. As an application we give a new approach to
  constructing a mean curvature flow with surgery for smooth compact
  initial surfaces in $\mathbb{R}^3$. 
\end{abstract}

\maketitle
\section{Introduction}
Let $S_t\subset\mathbb{R}^3$ be a family of compact surfaces evolving by
the mean curvature flow. For a given initial surface $S_0$ the flow
can encounter singularities, which in principle can be rather
complicated. However, an influential
conjecture of Huisken (see Ilmanen~\cite[Problem 8]{IlmanenProb}) states
that for \emph{generic} initial 
data, the flow should only encounter spherical and cylindrical
singularities. This conjecture was recently resolved by
Chodosh-Choi-Mantoulidis-Schulze~\cite{CCMS24, CCMS24_2}, building on
foundational earlier works by Colding-Minicozzi~\cite{CM12}, and the
solution of the multiplicity one conjecture by
Bamler-Kleiner~\cite{BK23}. An important consequence of this is that
by Hershkovits-White~\cite{HW19} and
Choi-Haslhofer-Hershkovits~\cite{CHH22} the flow with such generic 
initial data can be defined uniquely through the singularities in a
weak sense. 

Spherical singularities are modeled on the shrinking sphere, and they
are very well understood -- a connected component of the flow must
shrink to a point, and since the sphere is compact, near the
singularity in spacetime the flow of this component can be written as
a graphical flow over the sphere, see for instance Huisken~\cite{Huisken84}. In contrast,
cylindrical singularities are modeled on the self shrinking cylinder
$\sqrt{-t}\mathcal{C} = \mathbb{R} \times S^1(\sqrt{-2t})$ as $t\nearrow 0$, which is
noncompact, and this causes significant technical complications. For
instance cylindrical singularities may be non-isolated, as is the case
for the ``marriage ring'' given by a torus that collapses to a
circle. The work of Colding-Minicozzi~\cite{CM15}, and more recently
Sun-Wang-Xue~\cite{SWX25_2}, gives a detailed
analysis of the structure of the singular set near cylindrical
singularities. The upshot of these results is that
the cylindrical singularities are contained in the union of finitely
many $C^{2,\alpha}$ curves in spacetime.

To go further, it was conjectured by Ilmanen~\cite[Problem 9]{IlmanenProb} (see also
Colding-Minicozzi-Pedersen~\cite[Conjecture 7.1]{CMP15},
Sun-Xue~\cite[Conjecture 1.1]{SX22}) that, for
generic initial conditions, the situation is much better than this, and
that in fact the cylindrical singularities are isolated in spacetime. 
Our main result is to verify this conjecture up to the first singular
time. More precisely we show the following.
\begin{thm}\label{thm:main}
  Suppose that $S_0\subset \mathbb{R}^3$ is a smooth compact
  surface. There are arbitrarily small $C^2$ perturbations
  $\tilde{S}_0$ of $S_0$ such that the mean curvature flow
  $\tilde{S}_t$ with initial condition $\tilde{S}_0$ admits
  only spherical and nondegenerate cylindrical singularities at its
  first singular time. 
\end{thm}

The notion of \emph{nondegenerate} cylindrical singularity here is in
the sense studied by Angenent-Vel\'azquez~\cite{AV97}, and also more
recently Sun-Xue~\cite{SX22}. For
the precise definition see Section~\ref{sec:prelim}, however note that by
\cite{SX22}, nondegenerate cylindrical singularities are isolated in
spacetime  and stable under small perturbations. In particular it
follows that if $T_1$ is the first singular time of the perturbation
$\tilde{S}_t$, then there is some
$T_2> T_1$ such that the flow through singularities is smooth for
$t\in (T_1, T_2)$. 

Note that this result is new even in the setting of mean convex flows, in
which case White~\cite{White02, White15} showed that in any dimension
the flow can only
encounter multiplicity one spherical and cylindrical singularities.
Many of the techniques used in the proof of Theorem~\ref{thm:main} also
work for mean convex flows in higher dimensions, however a new
difficulty in that case is the possibility of cylindrical
singularities modeled on $S^{n-k}\times \mathbb{R}^k$ that are
degenerate in only certain 
directions in the $\mathbb{R}^k$ factor. We remark that if we make the
stronger assumption that the flow is 2-convex, then the only cylinder
that appears is $S^{n-1}\times \mathbb{R}$. We will extend the results
to the general mean convex setting in forthcoming work.  We also
expect that the conclusion of the theorem holds for all time, however this requires
a more detailed understanding of the behavior of the flow across
nondegenerate cylindrical singularities (see
Remark~\ref{rem:initialperturb}). 

An application of our result is a new construction of a mean
curvature flow with surgery for surfaces in $\mathbb{R}^3$. Recall
that Huisken-Sinestrari~\cite{HS08}, Brendle-Huisken~\cite{BH15} and
Haslhofer-Kleiner~\cite{HK17} developed a method that allows one to
continue a 2-convex mean curvature flow past cylindrical singularities, by
performing surgeries. A more 
recent approach relying less on detailed a priori estimates was given
by Haslhofer~\cite{Haslhofer25}. These approaches
rely on an understanding of the high curvature, or neck, regions near cylindrical
singularities, and showing that they admit canonical neighborhoods. These
neighborhoods can be
replaced by different geometric models, altering the topology of the
surface by cutting necks and capping off the resulting ends. Then the
flow can be continued. This procedure depends on choosing suitable
parameters to determine when to perform the surgeries in order to
bypass the singularities. Letting the surgery parameters pass to a 
suitable limit, the corresponding flows with surgeries with a given initial condition
$S_0$ converge to the original mean curvature flow $S_t$. 

If we knew that Theorem~\ref{thm:main} applies for all times, not just
the first singular time, then the results of
Sun-Wang-Xue~\cite{SWX25} provide an
alternative approach to mean curvature flow with surgery (see
\cite[Corollary 1.3]{SWX25}). Indeed,
\cite[Theorem 1.1]{SWX25} shows that when the 
flow passes through a nondegenerate cylindrical singularity, then the
flow itself performs the corresponding surgery. Although
Theorem~\ref{thm:main} applies only to the first singular time, we
will use the same methods to show the following.
\begin{thm}\label{thm:surgery}
  Let $S_0\subset \mathbb{R}^3$ be a smooth compact embedded surface,
  and let $S_t\subset \mathbb{R}^3$, for $t\in[0,\infty)$, be a mean
  curvature flow (more precisely a unit-regular, cyclic, integral
  Brakke flow) with initial condition $S_0$, admitting
  only spherical and cylindrical singularities. Then there exists a
  number $N > 0$, and a sequence $\tilde{S}^k_t$ of flows
  with surgeries converging to $S_t$ as $k\to\infty$ in the following
  sense. For each $k$ there are surgery times
  \[ 0 = \tilde{T}^k_0 < \tilde{T}^k_1 < \ldots <  \tilde{T}^k_N, \]
  and mean curvature flows $\tilde{S}^k_{i, t}$ on the
  intervals $[\tilde{T}^k_{i-1}, \tilde{T}^k_i]$ for $i=1,\ldots,
  N$ that are smooth near the endpoints. These satisfy
  \begin{itemize}
    \item The $\tilde{S}^k_{i, t}$ only have spherical and
      nondegenerate cylindrical singularities, and $\tilde{S}^k_{N,
        t}$ becomes extinct before time $\tilde{T}^k_N$. 
    \item At each surgery time $\tilde{T}^k_i$, let us write
      $\tilde{S}^k_{i, \pm}$ for the two one-sided limits. These are
      smooth, and $\tilde{S}^k_{i, +}$ has good graphicality over
      $\tilde{S}^k_{i, -}$ in the following sense: $\tilde{S}^k_{i,
        -}$ has second fundamental form bounded by a constant $A_0$
      (depending on $i, k$), and $\tilde{S}^k_{i, +}$ is
      $(A_0k)^{-1}$-graphical over $\tilde{S}^k_{i, -}$. 
     \item We define $\tilde{S}^k_t$ by concatenating the flows
       $\tilde{S}^k_{i,t}$. As $k\to \infty$, the (discontinuous)
       flows $\tilde{S}^k_t$ converge to
       $S_t$ in the sense of measures on $\mathbb{R}^3\times [0,\infty)$.  
   \end{itemize}
 \end{thm}

 This result should be compared to that of
 Daniels-Holgate~\cite{DH22}, who showed that mean curvature flows
 with surgery can be used to approximate flows with only spherical and
 neck pinch singularities.
 We emphasize, however, that our notion of a flow with surgery is different from
 the notions studied in \cite{HS08, BH15, HK17, DH22}. In these works the
 surgeries are performed near the singular times, and so the resulting flows
 ``jump over'' the singularities in a discontinuous way, while 
 still controlling the change in topology.  In contrast, the
 surgeries in the flows constructed by Theorem~\ref{thm:surgery} do
 not change the topology, but rather the smooth surface is perturbed
 slightly to a different smooth surface at the surgery times. The
 topological changes along the flow are performed by the flow itself, by
 passing through spherical and nondegenerate cylindrical
 singularities  in accordance with the work of
 Sun-Wang-Xue~\cite{SWX25}. We expect that these smooth perturbations
 are not actually required, but are not able to prove this at
 present.

Let us give a brief overview of the argument for perturbing away
degenerate singularities, that lies at the heart of
Theorems~\ref{thm:main} and \ref{thm:surgery}.
First, as mentioned above,
by \cite{CCMS24_2} we can assume that after a small perturbation our flow $S_t$ 
admits only spherical and cylindrical singularities, and our goal is
to perturb away the degenerate cylindrical singularities. Suppose 
that the flow $S_t$ has a degenerate cylindrical singularity at a point
$(X,T)$, and let $M_\tau$ denote the rescaled mean curvature flow
centered at $(X,T)$. Up to time translation (and rotation if
necessary) we can assume that on
large balls of radius $R(\tau) \to \infty$, the $M_\tau$ are graphs of
$v(\tau)$ over $\mathcal{C}$, and $\Vert v(\tau)\Vert_{L^2} \leq \eta_0
e^{-\tau/2}$ for a small $\eta_0 > 0$. In practice we will get slightly worse decay, but for
this discussion we ignore that. The main task is to show that the initial
condition $M_0$ has arbitrarily small perturbations so that the
corresponding unrescaled flows have no degenerate singularity in a parabolic ball
around $(0,0)$ of a definite size. In the setting of
Theorem~\ref{thm:main} one then needs to show that arbitrarily
small initial perturbations of the flow $S_t$ also exist with no
degenerate cylindrical singularities in a neighborhood of $(X,T)$. 
This last step is why the statement of Theorem~\ref{thm:main} is
restricted to the first singular time. 

To construct perturbations of the flow $M_\tau$, we consider a small parameter
$a$, and define $M^a_\tau$ 
to be the rescaled mean curvature flow whose initial condition is the
graph of $a \chi_{R(0)} y$, where $\chi_{R(0)}$ is a cutoff function
supported in the $R(0)$-ball, and $y$ is the coordinate along the
$\mathbb{R}$-factor of the cylinder $\mathcal{C}$. In practice we will
have a small perturbation of this. Let us also write
$L^a_t$ for the corresponding unrescaled flows, where $L^0_t$ is a
suitable rescaling of $S_t$, with a singularity at $(0,0)$. Naively we hope
that on the balls $B_{R(\tau)}$ the flow $M^a_\tau$ remains graphical
over $M_\tau$ at least for some time $\tau < T_a$, given by the graphs of
$u_a(\tau)$.  Moreover, one expects that the growth of $u_a(\tau)$ is
roughly bounded below by the solution $a y e^{\tau/2}$ of the
linearized rescaled mean curvature flow equation on $\mathcal{C}$.

In order to ensure that the perturbed flow $M^a_\tau$ does not have a
degenerate cylindrical singularity at infinity, we look at the time $\tau_0$ at
which $u_a(\tau)$ dominates $v(\tau)$, i.e. when we have $|a| e^{\tau_0/2} =
e^{-\tau_0/2}$, so $e^{\tau_0} = |a|^{-1}$. At this time we will find
that $M^a_\tau$ has a rate of growth not just relative to $M_\tau$,
but also relative to 
$\mathcal{C}$. It follows from the discrete frequency monotonicity
results of  \cite{SWX25} that this type of growth means that
$M^a_\tau$ cannot converge to the cylinder as $\tau\to \infty$. 
Crucially, if we translate $M^a_{\tau_0}$ in the $y$-direction by a
small amount $r_0 > 0$, the corresponding rescaled flow
still cannot converge to the cylinder.

Translating the rescaled flow by distance $r_0$ at time $\tau_0$ corresponds to a
translation by distance $r_0 e^{-\tau_0/2} = r_0 |a|^{1/2}$ at time
$0$. We will make this discussion 
precise below, and also incorporate rotations and translations in time
as well as translations transverse to the $y$-axis. The main conclusion will be
that the unrescaled flows $L^a_t$ have no cylindrical singularity with $|y| <
r_0|a|^{1/2}$ (in practice we will use $|a|^{2/3}$ instead). 
Unfortunately in itself this type of result is not what we want yet, since as $|a|\to 0$,
the set where $|y| < r_0 |a|^{1/2}$ is also shrinking.

The next observation is that if $L^a_t$ has a degenerate cylindrical
singularity with $y$-coordinate $|y_0| < r_0$, then the same type of argument implies
that $L^{a'}_t$ cannot have any cylindrical singularities with
$y$-coordinate $y_0'$ satisfying $|y_0 - y_0'| < |a-a'|^{1/2}$.
The upshot is that if $L^{a}_t$ and $L^{a'}_t$ have
degenerate cylindrical singularities with $y$-coordinates $y_0$
and $y_0'$ respectively, then $|y_0 - y_0'| \geq |a - a'|^{1/2}$. A simple
covering argument then implies that for a dense set
of parameters $a$ the perturbed flows $L^a_t$ have no degenerate cylindrical
singularity in a fixed parabolic neighborhood of $(0,0)$.

The main technical difficulty in executing this strategy is that we
need control of the perturbed flow $M^a_\tau$ for a sufficiently long
time. We will have to carefully choose the sizes $R(\tau)$ of the balls
that we are working on, and we will need to prove a three-annulus type result
for the graphicality function $u^a_\tau$ of $M^a_\tau$ over
$M_\tau$, in order to show that a rate of growth like $e^{\tau/2}$ at
$\tau=0$ persists for a sufficiently long time.
Since the behavior of the flow cannot be completely localized to the
$R(\tau)$-balls, a crucial ingredient is a global barrier argument, leading
to an estimate of the form $|u^a_\tau| \lesssim  a e^\tau$ on the $R(\tau)$-balls. 
Given this, the three-annulus lemma can be shown
similarly to the arguments in \cite{LSSz22} and Ghosh~\cite{Ghosh25}, by
proving a non-concentration estimate at infinity.

\section{Preliminary results}\label{sec:prelim}
In this section we will recall and prove some basic results which will
be used later. Throughout the paper we consider the
mean curvature flow of embedded closed hypersurfaces $S_t\subset
\mathbb{R}^3$. The flow can be defined for all time in a weak sense,
through singularities, as a unit regular, cyclic integral Brakke flow (see
Ilmanen~\cite{Ilmanen94}, White~\cite{White09}). 
By the results of Chodosh-Choi-Mantoulidis-Schulze~\cite{CCMS24_2} we
can assume that we have already 
slightly perturbed the initial surface $S_0$ in such a way that $S_t$
only has spherical and cylindrical singularities. Using
Hershkovits-White~\cite{HW19} and
Choi-Haslhofer-Hershkovits~\cite{CHH22}, this implies that the flow is
non-fattening. This further implies that if we consider a sequence of
initial conditions $S^i_0$ converging to $S_0$, then the corresponding
Brakke flows $S^i_t$ converge to $S_t$, and using
Bernstein-Wang~\cite[Corollary 1.2]{BW17}, for large $i$ these flows
also only have spherical and cylindrical singularities. 

In the arguments
below, various constants will depend on a bound for the area ratios of
$S_t$, or equivalently, on a bound on the entropy of $S_0$, which
implies similar bounds for any rescaling of any time slice of
$S_t$. We will therefore assume that uniform area ratio bounds hold for any
(rescaled) mean curvature flow considered below, and we will
not explicitly state the dependence of other constants on these
bounds. 

Let us write $\mathcal{C} = \mathbb{R} \times S^1(\sqrt{2}) \subset \mathbb{R}^3$
for the cylinder, so that $\sqrt{-t}\mathcal{C}$ is a solution of the
mean curvature flow for $t < 0$, with a cylindrical singularity at
$(0,0)$. We use coordinates $(y, z_1,z_2)$ on $\mathbb{R}^3$,
so that the circle $S^1(\sqrt{2})$ lies in the $z_1z_2$-plane. We
write $x = (y, z_1, z_2)$ for the position vector in
$\mathbb{R}^3$. We will consider surfaces $M$ that are small
perturbations of $\mathcal{C}$ on a ball $B_R:= \{x\,;\, |x| < R\}$. For this, we say
that $M$ is $\delta$-graphical over $\mathcal{C}$ on a ball $B_R$, if
the following hold: $B_R\cap M$ is relatively closed, and can
be parametrized as the normal graph over a region $\Omega \subset
\mathcal{C}$ with a function 
$v$ satisfying $|v| + |\nabla v| + |\nabla^2 v| < \delta$. We
similarly define when $M'$ is $\delta$-graphical over $M$ on a ball
$B_R$. In defining these graphs over surfaces close to a cylinder, we
use the outward pointing unit normals. When $M'$ is the normal graph
of a function $v$ over $M$, then we will identify functions on $M'$
with functions on $M$, using the identification $x\mapsto x +
v(x)\mathbf{n}$ for $x\in M$. 

Given a mean curvature flow $L_t$, following Huisken~\cite{Hui90} we define
the corresponding rescaled mean curvature flow $M_\tau$
centered at $(0,0)$ by
\[ M_\tau = e^{\tau/2} L_{-e^{-\tau}}. \]
More generally, the rescaled flow centered at $(X,T)$ is obtained by
first replacing $L_t$ by $L_{t+T}-X$. We say that the flow $L_t$ 
has a $\mathcal{C}$-singularity at $(X,T)$, if the rescaled flow
$M_\tau$ centered at $(X,T)$ 
converges to $\mathcal{C}$ smoothly on compact sets as
$\tau\to\infty$. It was shown by Colding-Minicozzi~\cite{CM15_1} that it is
enough for this to require that $M_{\tau_i} \to \mathcal{C}$ on compact
sets along a sequence $\tau_i\to \infty$, and we will also say in this
case that $M_\tau$ has a $\mathcal{C}$-singularity at infinity.
More generally, we say that the flow has a cylindrical singularity at $(X,T)$, if the
corresponding rescaled flow converges to $Q\mathcal{C}$ for a rotation
$Q$. 

We next recall the notion of a nondegenerate cylindrical singularity
from Sun-Xue~\cite{SX22}, and Sun-Wang-Xue~\cite{SWX25_2}.
In \cite[Theorem 1.4]{SWX25_2} it was shown
that if the rescaled flow $M_\tau$ has a $\mathcal{C}$-singularity at
infinity, then one of the following happens:
\begin{itemize}
  \item[(a)] For some $K > 0$, and sufficiently large $\tau$, the $M_\tau$ are graphical over
    $\mathcal{C}$ on $B_{K\sqrt{\tau}}$, and the graphicality function
    $v(x,\tau)$ satisfies
    \[ \label{eq:nondegasympt} \left\Vert v(\cdot, \tau) - \frac{1}{\sqrt{2}\tau} (y^2 -
        2)\right\Vert_{L^2} = o(\tau^{-1}), \]
  \item[(b)] For some $K > 0$ and sufficiently large $\tau$, the
    $M_\tau$ are graphical over $\mathcal{C}$ on $B_{e^{K\tau}}$, and
    the graphicality function $v$ satisfies
    \[ \Vert v(\cdot, \tau) \Vert_{L^2} = O(e^{-\tau/2}). \]
  \end{itemize}
Here, as well as throughout the paper, $L^2$-norms are always taken
with respect to the Gaussian measure $e^{-|x|^2/4}$ as in
\eqref{eq:dCdefn} below. 
Note that the result in \cite{SWX25_2} holds in higher dimensions, and
is more refined. In our setting we combined their cases (ii) and
(iii), and the ``low spherical flow'' that appears is necessarily
trivial. We say that the $\mathcal{C}$-singularity is \emph{nondegenerate}, if
the case (a) holds. Otherwise we say that the
$\mathcal{C}$-singularity is degenerate. 

It will be useful for us to give an alternative characterization,
essentially implicit in \cite{SWX25_2}. For this, let us
define the following $L^2$-distance: 
\[\label{eq:dCdefn} \mathbf{d}_{\mathcal{C}}(M)^2 = \Vert \dist\Vert^2_{L^2(M)}
  = \int_M \dist(x)^2 e^{-|x|^2 / 4}\,
  d\mathcal{H}^2, \]
where $\dist(x) = \min\{\mathrm{dist}(x, \mathcal{C}), 1\}$ is a truncated
distance function from the cylinder. The following nonconcentration
estimate was shown in \cite{SWX25}.

\begin{prop}[See Corollary 3.3 in \cite{SWX25}] \label{prop:SWXnonconc}
  There is a constant $C$ with the following
  property. If $M_\tau$ is a rescaled mean curvature flow then for $\tau\in [0,2]$ we have
  \[ \int_{M_\tau}\overline{\mathrm{dist}}(x)^2 ( 1+ \tau |x|^2)\, e^{-|x|^2/4}\,
    d\mathcal{H}^2 \leq C \int_{M_0} \overline{\mathrm{dist}}(x)^2\,
    e^{-|x|^2/4}\, d\mathcal{H}^2. \]
\end{prop}

The following is essentially \cite[Remark 3.8]{SWX25}. 
\begin{prop}\label{prop:3ann2}\label{prop:L23ann} 
  Let $\lambda \not\in \frac{1}{2}\mathbb{Z}$.
  There exist $L_0, \delta > 0$, depending on $\lambda$, such
  that if $L \geq L_0$ and  $M_\tau$ is a rescaled flow that is $\delta$-graphical over
  $\mathcal{C}$ on $B_{\delta^{-1}}$ for $\tau\in [0,2L]$, then
  \[ \mathbf{d}_{\mathcal{C}}(M_L) \geq e^{\lambda L}
    \mathbf{d}_{\mathcal{C}}(M_0), \]
  implies
  \[ \mathbf{d}_{\mathcal{C}}(M_{2L}) \geq e^{\lambda L}
    \mathbf{d}_{\mathcal{C}}(M_L).  \]
\end{prop}
  Note that in \cite[Remark 3.8]{SWX25} this result is stated for
  $L=1$, which is a stronger result. For that we may need to define
  the distance function more   carefully, in terms of the function
  $D_{n,k}$ used in \cite{SWX25}. 
\begin{proof}
  Let $L > 0$, and suppose that we have a sequence of
  flows $M^i_\tau$ converging to $\mathcal{C}$ locally smoothly
  on compact subsets of $(0,2L]\times \mathbb{R}^{3}$ such that
  \[ \mathbf{d}_{\mathcal{C}}(M^i_L) \geq e^{L\lambda}
    \mathbf{d}_{\mathcal{C}}(M^i_0), \text{ and }
    \mathbf{d}_{\mathcal{C}}(M^i_{2L}) \leq e^{L\lambda}
    \mathbf{d}_{\mathcal{C}}(M^i_L). \]
   We will show that if $L$ is large enough, then this is a
   contradiction. We can write the $M^i_\tau$ as graphs of $u^i(x,
   \tau)$ over larger and larger subsets of $\mathcal{C}$, and
   if we let $\mathbf{d}_i = \mathbf{d}_{\mathcal{C}}(M^i_0)$,
   then $\mathbf{d}_i^{-1}u^i$ converges locally smoothly on compact
   subsets of $(0,2L]\times\mathbb{R}^{n+1}$ to a solution $u^\infty$
   of the linearized equation $\partial_\tau u^\infty = \mathcal{L}
   u^\infty$ on $\mathbb{C}$. Note that we may not have
   convergence at $\tau=0$. However from
   Proposition~\ref{prop:SWXnonconc} we have
   $\mathbf{d}_{\mathcal{C}}(M^i_1) \leq C
   \mathbf{d}_{\mathcal{C}}(M^i_0)$, and so from our hypothesis
   we have
   \[\mathbf{d}_{\mathcal{C}}(M^i_L) \geq e^{L\lambda}C^{-1}
     \mathbf{d}_{\mathcal{C}}(M^i_1). \]
   Using Proposition~\ref{prop:SWXnonconc} again, we find that
   \[ \Vert u^\infty(\cdot, L)\Vert_{L^2} \geq e^{L\lambda} C^{-1}
     \Vert u^\infty(\cdot, 1)\Vert_{L^2}, \text{ and } \Vert
     u^\infty(\cdot, 2L)\Vert_{L^2} \leq e^{L\lambda} 
     \Vert u^\infty(\cdot, L)\Vert_{L^2}. \]
  Note, however, that since there are no homogeneous solutions of the
  linearized equation with growth rate $\lambda$, there exists some
  $\lambda' < \lambda$ such that if
  $\Vert
     u^\infty(\cdot, L)\Vert_{L^2} \geq e^{-L\lambda} 
     \Vert u^\infty(\cdot, 2L)\Vert_{L^2}$, then $\Vert
     u^\infty(\cdot, 1)\Vert_{L^2} \geq e^{-(L-1)\lambda'} 
     \Vert u^\infty(\cdot, L)\Vert_{L^2}$. If we choose $L$ large
     enough so that $e^{(L-1)\lambda'}< e^{L\lambda}C^{-1}$, then we
     get the required contradiction. 
\end{proof}

While we do not need this,
using this type of result one can show that if $M_\tau$ has a $\mathcal{C}$-singularity
at infinity, then the limit
\[ \label{eq:growthlimit} \lim_{\tau\to\infty} \log \left(\frac{
      \mathbf{d}_{\mathcal{C}}(M_\tau)}{\mathbf{d}_{\mathcal{C}}
      (M_{\tau+1})}\right) \in \mathbb{R}\cup\{\infty\} \]
exists, and in fact (see \cite[Remark 3.8]{SWX25}) if the limit is
finite, then it equals a non-negative eigenvalue of the
linearized operator $-\mathcal{L}_{\mathcal{C}}$ that we will discuss
below. Since these eigenvalues lie in $\frac{1}{2}\mathbb{Z}$,
it follows that in the dichotomy above the nondegenerate case
is characterized by the limit in \eqref{eq:growthlimit} being 0, while
in the degenerate case the limit is at least $1/2$. A consequence of
this is the following, which gives us a more quantitative decay estimate for
degenerate cylindrical singularities.
\begin{prop}\label{prop:degendecay}
  Let $\eta_0, \kappa_1 > 0$. There exists a $\delta > 0$ depending on
  $\eta_0, \kappa_1$ with the following property. Suppose that
  $M_\tau$ is a rescaled mean curvature flow converging to
  $\mathcal{C}$ as $\tau\to\infty$, and $M_0$ is a $\delta$-graph over
  $\mathcal{C}$ on the ball $B_{\delta^{-1}}$. Assume moreover that the
  singularity at infinity is degenerate. Then we have
  $\mathbf{d}_{\mathcal{C}}(M_\tau) < \eta_0
  e^{-(\frac{1}{2}-\kappa_1)\tau}$ for all $\tau \geq 0$. 
\end{prop}
\begin{proof}
  Let $\kappa_1 \in (0,1/2)$. Given any $\delta' > 0$, by the work of
  Colding-Minicozzi~\cite{CM15_1} there is a $\delta > 0$ such that if
  $M_0$ is a $\delta$-graph over $\mathcal{C}$ on $B_{\delta^{-1}}$,
  then for all $\tau > 0$ we have that $M_\tau$ is a $\delta'$-graph
  over $\mathcal{C}$ on $B_{\delta'^{-1}}$. Let $L_0$ be the number from
  Proposition~\ref{prop:L23ann}, corresponding
  to $\gamma=\kappa_1-\frac{1}{2}$. Then if we
  choose $\delta'$ sufficiently small given $\kappa_1$, we have the
  following: if $\mathbf{d}_{\mathcal{C}}(M_{\tau+L_0}) \geq
  e^{(-\frac{1}{2}+\kappa_1)L_0} \mathbf{d}_{\mathcal{C}}(M_\tau)$ for some
  $\tau$, then also $\mathbf{d}_{\mathcal{C}}(M_{\tau+(k+1)L_0}) \geq
  e^{(-\frac{1}{2}+ \kappa_1)L_0}\mathbf{d}_{\mathcal{C}}(M_{\tau+k L_0})$ . This
  implies that the cylindrical singularity is nondegenerate. So in
  the degenerate case we must have $\mathbf{d}_{\mathcal{C}}(M_{\tau+L_0}) \leq
  e^{(-\frac{1}{2}+\kappa_1)L_0} \mathbf{d}_{\mathcal{C}}(M_\tau)$ for all
  $\tau$. Using Proposition~\ref{prop:SWXnonconc} it follows that we
  have $\mathbf{d}_{\mathcal{C}}(M_\tau) \leq C_{L_0}
  e^{(-\frac{1}{2}+\kappa_1)\tau}\mathbf{d}_{\mathcal{C}}(M_0)$ for
  all $\tau \geq 0$. By choosing $\delta$ sufficiently small, we can ensure that
 $C_{L_0} \mathbf{d}_{\mathcal{C}}(M_0) < \eta_0$.
\end{proof}

The following is another simple consequence of the frequency
monotonicity, that we will need later. 
\begin{prop}\label{prop:noC}
  Let $\kappa > 0$. There exist $L_0, \delta > 0$ such that if $M_\tau$
  has a cylindrical singularity at infinity (i.e. converges to
  $Q\mathcal{C}$ for a rotation $Q$), and in addition $M_0$
  is $\delta$-graphical over $\mathcal{C}$ on $B_{\delta^{-1}}$, then
  we have
  \[ \mathbf{d}_{\mathcal{C}}(M_{\tau+L_0}) \leq e^{\kappa L_0}
    \mathbf{d}_{\mathcal{C}}(M_\tau), \]
  for all $\tau\geq 0$. 
\end{prop}
\begin{proof}
  Without loss of generality we can assume that $\kappa\in (0,1/2)$. 
  Given $\kappa$, consider the $\delta$ determined in
  Proposition~\ref{prop:L23ann} by $\gamma=\kappa$. By
  Colding-Minicozzi~\cite{CM15_1}, if we assume that $M_0$ is 
  $\delta'$-graphical over $\mathcal{C}$ on $B_{\delta'^{-1}}$ for
  sufficiently small $\delta'$, and $M_\tau$ has a cylindrical
  singularity at infinity (a rotation of $\mathcal{C}$), then we have
  that $M_\tau$ is $\delta$-graphical over $B_{\delta^{-1}}$ for all
  $\tau \geq 0$. If we then have
  $\mathbf{d}_{\mathcal{C}}(M_{\tau_0+L_0}) \geq e^{\kappa L_0}
    \mathbf{d}_{\mathcal{C}}(M_{\tau_0})$ for some $\tau_0$, it
    follows by using Proposition~\ref{prop:L23ann} repeatedly
    that $\mathbf{d}_{\mathcal{C}}(M_{\tau_0 + kL_0}) \geq e^{k\kappa L_0}
    \mathbf{d}_{\mathcal{C}}(M_{\tau_0})$ for all $k\geq 1$. This is
    a contradiction because then $M_\tau$ cannot remain
    $\delta$-graphical over $\mathcal{C}$ on $B_{\delta^{-1}}$ for all
    $\tau$. 
\end{proof}

We next consider Jacobi fields on $\mathcal{C}$, i.e. solutions of the
linearization of the rescaled mean curvature flow equation $\partial_\tau v =
\mathcal{L}_{\mathcal{C}} v$, where
\[ \mathcal{L}_{\mathcal{C}} v &= \Delta v + \frac{1}{2}(v - x\cdot
  \nabla v) + |A|^2 v \\
  &=\Delta v - \frac{1}{2} y \partial_y v
  + v. \]
Such Jacobi fields can be used to model the behavior of a rescaled
flow $M_\tau$ close to $\mathcal{C}$. 
Any eigenfunction of $\mathcal{L}_{\mathcal{C}}$ satisfying
$\mathcal{L}_{\mathcal{C}} v = \lambda v$ gives rise to a solution
$e^{\lambda \tau} v$ of the linearized equation. It is well known (see
e.g. \cite[Section 2.1]{SX22}) that the eigenvalues are
$\frac{1}{2}\mathbb{Z} \cap (-\infty, 1]$, and the
eigenfunctions with the largest eigenvalues are
given by:
\begin{itemize}
  \item $\lambda=1$: spanned by $v=1$, corresponding to translation of
    the (unrescaled) flow in time.
  \item $\lambda=\frac{1}{2}:$ spanned by linear functions on
    $\mathbb{R}^3$ restricted to $\mathcal{C}$. Those spanned by
    $z_1,z_2$ correspond to translations of the cylinder orthogonal to
    the $\mathcal{R}$-factor. The eigenfunction $y$ is more subtle,
    since $\mathcal{C}$ is translation invariant in the
    $y$-direction. As discussed in \cite{SX22}, the eigenfunction $y$
    arises geometrically from
    $y$-translations of nondegenerate singularities. It will play a key
    role in this paper. 
  \item $\lambda=0$: spanned by $z_iy$, corresponding to rotations,
    and $y^2-2$, which is the ``non-integrable'' Jacobi field.
  \end{itemize}
It is natural to expect that generically the flow should encounter
singularities modeled on the ``least decaying'' Jacobi fields that
do not correspond to symmetries. In our case this is $y^2-2$, and
indeed this arises in the asymptotics \eqref{eq:nondegasympt}. The
decaying asymptotics in the ``degenerate'' alternative (b) corresponds
to the fact that other Jacobi fields decay at least as $e^{-\tau/2}$.

It will be convenient to record the following results, making a more explicit
connection between the Jacobi fields above, and the corresponding
symmetries.

\begin{lemma}\label{lem:geomJacobi}
  Suppose that the $L^i_t$ are mean curvature flows for $t\in [-2,0)$,
  with corresponding rescaled flows $M^i_\tau$ for $\tau\in [-\ln 2,
  \infty)$. Suppose that the $M^i_\tau$
  converge smoothly on compact subsets of $[-\ln 2 ,\infty) \times
  \mathbb{R}^3$ to $\mathcal{C}$. In addition, let $s_i, \beta_{1,i},
  \beta_{2,i} \in \mathbb{R}$ and $Q_i \in SO(3)$ be rotations with axes
  orthogonal to the $y$-axis, such that
    \[ |s_i| + |\beta_{1,i}| + |\beta_{2,i}| + |Q_i - Id| = \gamma_i
      \to 0. \]
   Consider the transformed flows
  \[  \tilde{L}^i_t := Q_i L^i_{t+s_i} + (0, \beta_{1,i},
    \beta_{2,i}), \]
  and let $N^i_\tau$ be the corresponding rescaled
  flows.

  Fixing a large $R > 0$, for large $i$ we can write
  $N^i_\tau$ as the graph of $v_i(x,\tau)$ over $M^i_\tau$ on
  $[0,R^2] \times B_R$. These functions can also be viewed as functions
  on $\mathcal{C}$. Then along a subsequence
  \[ \lim_{i\to \infty} \gamma_i^{-1} v_i(x, \tau) = b e^\tau + (c_1
    z_1 + c_2 z_2) e^{\tau/2} + (d_1 z_1 y + d_2 z_2 y) \]
  smoothly on $[0,R^2] \times B_R$, where $b, c_k, d_k$ are suitable
  constants. Moreover we have
  \[ C^{-1} < |b| + |c_1| + |c_2| + |d_1| + |d_2| < C \]
  for a fixed constant $C$. 
\end{lemma}
\begin{proof}
  This result is quite standard, but for the convenience of the reader
  we give some details.  Note that by a straightforward calculation we have
  \[ N^i_\tau = (1 - e^\tau s_i)^{1/2} Q_i M^i_{\tau - \log(1-e^\tau
      s_i)} + e^{\tau/2} (0, \beta_{1,i}, \beta_{2, i}). \]
  We fix a large $R$, and consider $N^i_\tau$ for large $i$ in the region $[0,R^2]\times
  B_R$.  First suppose that $M_\tau^i =
  \mathcal{C}$ for all $\tau, i$, so
  \[ N^i_\tau = (1 - e^\tau s_i)^{1/2} Q_i \mathcal{C} + e^{\tau/2}
    (0, \beta_{1,i}, \beta_{2, i}). \]
  We can write the rotation matrix $Q_i = e^{A_i}$, where
  \[ A_i = \begin{pmatrix}
      0 & -d_{2,i} & d_{1,i} \\
      d_{2,i} & 0 & 0 \\
      -d_{1,i} & 0 & 0 \end{pmatrix}, \]
  and we can choose the norm so that $|Q_i - Id| = |d_{1,i}| +
  |d_{2,i}|$. At a point $(y, z_1,z_2)\in \mathcal{C}$ the unit normal vector is
  $2^{-1/2} (0,z_1, z_2)$, and the vector induced by the infinitesimal
  rotation $A_i$ is
  \[ A_i \begin{pmatrix} y\\ z_1 \\z_2 \end{pmatrix} = \begin{pmatrix}
      -d_{2,i}z_1 + d_{1,i}z_2 \\ d_{2,i}y \\ - d_{1,i}
      y \end{pmatrix}. \]
  The normal component of this is $2^{-1/2}(d_{2,i} z_1y - d_{1,i}
  z_2y)$. 
  Therefore for large $i$ we can write $Q_i\mathcal{C}$ as the graph
  of the function $2^{-1/2}(d_{2,i} z_1y - d_{1,i}
  z_2y) + O(\gamma_i^2)$ over $\mathcal{C}$, on the region
  $[0,R^2]\times B_R$. Considering also the scaling by $(1-e^\tau
  s_i)^{1/2}$ and the translation by $e^{\tau/2}(0,\beta_{1,i},
  \beta_{2,i})$ we find that $N^i_\tau$ is the graph of the function
  $v_i$ over $\mathcal{C}$, where
  \[ v_i(y, z_1,z_2) = -\frac{\sqrt{2}}{2} e^\tau s_i + 2^{-1/2} e^{\tau/2}
    (\beta_{1,i}z_1 + \beta_{2,i}z_2) + 2^{-1/2}(d_{2,i} z_1y - d_{1,i}
  z_2y) + O(\gamma_i^2). \]
  Note that the constant in the $O(\gamma_i^2)$ term depends on $R$,
  but for fixed $R$, if we let $i\to\infty$, then we get the required
  result, up to relabeling the constants. 

  Consider now the more general setting, where the $M^i_\tau$ are
  $\delta_i$-graphical over $\mathcal{C}$ on $[0,R^2]\times B_R$, with
  $\delta_i\to 0$. Then in the discussion above we obtain additional
  errors of order $\delta_i\gamma_i$. After dividing by $\gamma_i$,
  these still converge to zero as $i\to\infty$, so we obtain the same
  result. 
\end{proof}

\begin{lemma}\label{lem:geomJacobi2}
  There are $\epsilon, C > 0$ with the following property. 
  Suppose that $M_\tau$ is a rescaled mean curvature flow
  that is $100^{-1}$-graphical over $\mathcal{C}$ on the
  ball $B_R$, with $R > 10$, for $\tau\in [T-1,T+1]$ for some $T > 1$. Let $s\in
  \mathbb{R}$, $x_0\in \mathbb{R}^3$ and $Q\in SO(3)$ such that
  \[ R e^T |s| + e^{T/2} |x_0| + R |Q - Id| < \epsilon. \]
  Consider the rescaled mean curvature flow $N_\tau$ defined by
  \[\label{eq:Ntaudefn}
    N_\tau = (1-e^\tau s)^{1/2} QM_{\tau - \log(1-e^\tau s)} + e^{\tau/2}
    x_0, \]
  for $\tau\in [T-1/2, T +1/2]$. Then $N_\tau$ is $100^{-1}$-graphical
  over $M_\tau$ on $B_{R-1}$, where the graphicality function $v$
  satisfies
  \[ |v(x, \tau)| < C (R e^\tau |s| + e^{\tau/2} |x_0| + R|Q - Id|). \] 
\end{lemma}
\begin{proof}
This result holds more generally, and only relies on the fact that by
our assumptions $M_\tau$ has bounded geometry in $[T-1,T+1]\times B_R$. From this it
follows first that if $\epsilon$ is sufficiently small, then $M_{\tau -
  \log(1-e^\tau s)}$ is $C e^T|s|$-graphical over $M_\tau$. Given a
point $x\in B_R$, we have
\[ |\Big((1 - e^\tau s)^{1/2}Q x + e^{\tau/2} x_0\Big) - x| & \leq
  C|Q-Id| |x|+ e^\tau |s| |x| + e^{\tau/2} |x_0| \\
  &\leq C( Re^\tau |s| + e^{\tau/2} |x_0| + R|Q-Id|) \]
for $\tau\in [T-1,T+1]$, if $\epsilon$ is sufficiently small. The
required result follows from this. 
\end{proof}

The following is a consequence of pseudolocality (see 
  Ilmanen-Neves-Schulze~\cite[Theorem 1.5]{INS19}), and  the interior
  estimates of Ecker-Huisken~\cite{EH91}. See Sun-Xue~\cite[Theorem
  2.4]{SX22} for a proof. 
\begin{prop}\label{prop:pseudo}
  Given $\delta_0 > 0$ there exist $\delta_1, C_1 > 0$, depending on
  $\delta_0$ (and the area ratio bounds)
  satisfying the following. Suppose that $M_\tau$ is a
  rescaled mean curvature flow such that $M_0$ is a $\delta_1$-graph
  over $\mathcal{C}$ on the ball $B_R$ for some $R > C_1$. Then for
  $\tau\in [0,10]$, $M_\tau$ is a $\delta_0$-graph over $\mathcal{C}$
  on the ball $B_{e^{\tau/2}(R-C_1)}$. 
\end{prop}

We will need the following non-concentration estimate. The basic idea
appears in \cite{LSSz22}, using Ecker's log-Sobolev
inequality~\cite{Ecker}, and similar estimates were also used in  
\cite{LSz24,Ghosh25}, except we need to be more careful about the
error obtained from outside of the graphical region. Note that a
different non-concentration estimate was shown in \cite{SWX25,SWX25_2},
which has the advantage of being more global, but requires the
``reference'' flow to be a generalized cylinder, which is too
restrictive for our application. 

\begin{prop}\label{prop:nonconc}
  Suppose that $\delta_0 > 0$ is sufficiently small, and $\delta_1,
  C_1$ are determined by Proposition~\ref{prop:pseudo}. Let $R >
  2C_1$. Suppose that $M_\tau, N_\tau$ are two rescaled mean curvature
  flows, such that $M_0, N_0$ are $\delta_1$-graphs over $\mathcal{C}$
  on the ball $B_R$.
  By Proposition~\ref{prop:pseudo} we can then
  write $N_\tau$ as the graph of $v(x, \tau)$ over $M_\tau$ on the ball
  $B_{e^{\tau/2}(R-C_1)}$ for $\tau\in [0,2]$, with $|v(x,\tau)| <
  3\delta_0$. Suppose that we have the potentially better bound 
  $|v(x,\tau)| < \delta$ for $\tau\in [0,2]$ on these balls.

  There is a $c > 0$, and given $\kappa_3 > 0$, there exists $C_{\kappa_3} > 0$ (depending on
  $\kappa_3$) such that for $\tau\in [0,2]$ we have
  \[ \label{eq:L2growth} \Vert v(\cdot, \tau)\Vert^2_{L^2(B_R)} \leq
    e^{c\tau} \Vert v(\cdot,
    0)\Vert^2_{L^2(B_R)} + C_{\kappa_3} \delta^2e^{-(R-C_1-1)^2/
      (4+\kappa_3)}. \]
  In addition there exists $p_0 > 1$ (to be specific, we can choose
  $p_0=1.4$), such that for $\tau\in [1,2]$ we  have 
  \[ |\nabla^k v(x, \tau)| \leq C_{\kappa_3} \Big(
    \int_{M_0\cap B_R} |v(x,0)|^2\, e^{-|x|^2/4}\, d\mathcal{H}^2 +
    \delta^2e^{-(R-C_1-1)^2/ (4+\kappa_3)}\Big)^{1/2} e^{|x|^2 / 8p_0}, \]
  for $k=0,1,2$, and $|x| < e^{\tau/2}(R-C_1-2)$. 
\end{prop}
\begin{proof}
  The proof follows the argument in the proof of \cite[Lemma
  34]{LSz24}. For the convenience of the reader we include it here
  since the statement is not quite the same. 

  On the balls $B_{e^{\tau/2}(R-C_1)}$ we view $N_\tau$ as the graph
  of $v(x,\tau)$ over $M_\tau$. Using that both satisfy the rescaled
  mean curvature flow equation, the function $v$ satisfies an equation
  of the form 
  \[ \partial_\tau v = \Delta v + |A|^2 v + \frac{1}{2}(v - x\cdot
    \nabla v) + Q(x,\tau, v, \nabla v, \nabla^2 v). \]
  Here $A$ is the second fundamental form of $M_\tau$, and for each
  $x\in M_\tau$, the function $Q(x,\tau, p,q,r)$ is a power series in
  $p,q,r$ with terms that are at least quadratic, but $r$ appears at
  most linearly. The coefficients are controlled
  uniformly in $x,\tau$. Note that since $M_\tau$ is a
  $\delta_0$-graph over $\mathcal{C}$ on the relevant region,
  we can assume that $|A|^2 < 1$.  It follows that once $\delta_0$ is
  sufficiently small, we have the differential inequality
  \[  \left| \partial_\tau v - \Delta v + \frac{1}{2}x\cdot \nabla
      v\right| \leq C(|v| + |\nabla v|), \]
  for a fixed constant $C$. We absorb the gradient term by considering
  the evolution of $|v|^{3/2}$, which for another constant $C$
  satisfies
  \[ \partial_\tau |v|^{3/2} - \Delta |v|^{3/2} + \frac{1}{2} x\cdot
    \nabla |v|^{3/2} \leq C |v|^{3/2}, \]
  in a weak sense. 
  It follows that for a suitable constant $C_2 > 0$ the function
  $f(x,\tau) := e^{-C_2\tau} |v|^{3/2}$ is a subsolution of the drift heat equation
  along $M_\tau$, in the ball $B_{e^{\tau/2}(R-C_1)}$,
  i.e. $\partial_\tau f \leq \Delta f - \frac{1}{2}x\cdot \nabla f$. Note that $f
  \leq \delta^{3/2}$. 

  Next we use that $e^{-\tau} |x|^2$ is also a subsolution of the
  drift heat equation along any rescaled mean curvature flow. Define
  the function $\tilde{f}$ along $M_\tau$, by
  \[ \tilde{f}(x,\tau) = \begin{cases} \delta^{3/2}(e^{-\tau}|x|^2 -
      (R-C_1-1)^2)\, \quad &|x| \geq e^{\tau/2}(R-C_1), \\
      \max\{ f, \delta^{3/2}(e^{-\tau}|x|^2 -
      (R-C_1-1)^2)\, \quad &|x| < e^{\tau/2}(R-C_1). \end{cases} \]
  Note that if $|x| \leq e^{\tau/2}(R-C_1-1)$, then we have $\tilde{f}
  = f$, while for $|x|$ close to $e^{-\tau/2}(R-C_1)$ we have
  $\tilde{f} = \delta^{3/2}(e^{-\tau}|x|^2 -
      (R-C_1-1)^2)$ using that $f \leq \delta^{3/2}$. It follows that
      $\tilde{f}$ is also a subsolution of the drift heat equation
      along $M_\tau$. At the same time, for $\tau\in  [0,2]$ we have
      \[ \label{eq:tildef43}
        \tilde{f}^{4/3} \begin{cases} \leq \delta^2 |x|^{8/3}, \quad
        &|x| \geq e^{\tau/2}(R-C_1), \\
       \leq |v|^2 + \delta^2|x|^{8/3}, &e^{\tau/2}(R-C_1-1) < |x| <
       e^{\tau/2} (R-C_1), \\
       = e^{-4C_2\tau/3} |v|^2, &|x| \leq e^{\tau/2}(R-C_1-1). 
      \end{cases}
    \]
    Using this, and the assumed uniform area ratios, we have
    \[ \int_{M_0} \tilde{f}(x, 0)^{4/3}\, e^{-|x|^2/4}\,
      d\mathcal{H}^2 \leq A,\]
    where
    \[ \label{eq:Adefn} A &= \int_{M_0 \cap B_R} |v(x, 0)|^2 e^{-|x|^2/4}\,
      d\mathcal{H}^2 + \delta^2\int_{M_0 \setminus B_{R-C_1-1}}
      |x|^{8/3} e^{-|x|^2/4}\, d\mathcal{H}^2 \\
      &\leq \int_{M_0 \cap B_R} |v(x, 0)|^2 e^{-|x|^2/4}\,
      d\mathcal{H}^2 + \delta^2 C_{\kappa_3}
      e^{-(R-C_1-1)^2/(4+\kappa_3)}. \]
    The constant $C_{\kappa_3}$ here depends on a choice of $\kappa_3
    > 0$ and the area ratio bounds.

    To obtain \eqref{eq:L2growth} we apply the monotonicity formula,
    which implies that
    \[ \label{eq:62} \int_{M_\tau} \tilde{f}(x, \tau)^{4/3} e^{-|x|^2/4}\,
      d\mathcal{H}^2 \leq \int_{M_0} \tilde{f}(x,0)^{4/3} e^{-|x|^2/4}\,
      d\mathcal{H}^2 \leq A. \]
    From \eqref{eq:tildef43} we get
    \[ \label{eq:61} \int_{M_\tau} \tilde{f}(x, \tau)^{4/3} e^{-|x|^2/4}\,
      d\mathcal{H}^2 \geq e^{-c\tau} \int_{M_\tau \cap B_{R-C_1-1}}
      |v(x,\tau)|^2 e^{-|x|^2/4}\, d\mathcal{H}^2, \]
    and estimating the region outside of the $R-C_1-1$ ball as in
    \eqref{eq:Adefn} we have
    \[e^{-c\tau} \int_{M_\tau \cap B_{R-C_1-1}}
      |v(x,\tau)|^2 e^{-|x|^2/4}\, d\mathcal{H}^2 \geq e^{-c\tau}
      \Vert v\Vert_{L^2(B_R)}^2 - \delta^2 C_{\kappa_3}
      e^{-(R-C_1-1)^2/(4+\kappa_3)}. \]
    Combining this with \eqref{eq:61}, \eqref{eq:62} and
    \eqref{eq:Adefn}, we get the estimate \eqref{eq:L2growth}. 
    
    From Ecker's log-Sobolev inequality~\cite[Theorem 3.4]{Ecker}, we
    have a $q > 1$ such that for all $\tau\in [\frac{1}{2}, 2]$ we
    have
    \[ \int_{M_\tau} \tilde{f}(x, \tau)^{4q/3}\, e^{-|x|^2/4}\,
      d\mathcal{H}^2 \leq CA^q, \]
    for a larger constant $C$. In the notation of \cite{Ecker} we can
    set $p(0)=4/3$, so $p(1/2)=1+e/3 > 1.4 p(0)$, so we can choose
    $q=1.4$. Using the monotonicity formula centered
    at different points as in the proof of \cite[Lemma
    3.5(2)]{LSSz22}, this integral estimate implies the pointwise bound
    \[ \tilde{f}(x, \tau)^{4q/3} \leq CA^q e^{|x|^2/4},  \]
    for $\tau\in [3/4, 2]$. This, together with \eqref{eq:tildef43} and interior
    estimates, in turn implies the required
    pointwise bounds for $\nabla^k v$ for $\tau\in [1,2]$ and $|x| <
    e^{\tau/2}(R-C_1-2)$.

\end{proof}

The final ingredient that we need is the following three annulus type
lemma, similar to the ones used in \cite{LSSz22}, \cite{Ghosh25}. 
\begin{prop}\label{prop:3ann}
  Let $\lambda_1 \not\in \frac{1}{2}\mathbb{Z}$ and $\kappa_3 > 0$. There
  are $\lambda_2 > \lambda_1$ and
  $R_2, \delta_2 > 0$ (depending on $\lambda_1, \kappa_3$) with the 
   following property. Suppose that $M_\tau, N_\tau$ are two rescaled
   mean curvature flows for $\tau\in [0,2]$ that are
   $\delta_2$-graphical over $\mathcal{C}$ on the ball $B_R$ with $R >
   R_2$, and
   satisfy:
   \begin{itemize}
     \item[(a)] On $B_R$, $N_\tau$ is the graph of $v$ over $M_\tau$, where
       $|v|<\delta < \delta_2$,
     \item[(b)] We have $\Vert v(1)\Vert_{B_R} \geq e^{\lambda_1}
       \Vert v(0)\Vert_{B_R}$,
     \item[(c)] We have
       \[ \Vert v(1)\Vert_{B_R} \geq \delta e^{-\frac{R^2}{8 +
             2\kappa_3}}. \]
     \end{itemize}
     Then we have
     \[ \Vert v(2) \Vert_{B_R} \geq e^{\lambda_2} \Vert
       v(1)\Vert_{B_R}. \]
     Here by $\Vert v(\tau)\Vert_{B_R}$ we mean the (Gaussian) $L^2$-norm
     of $v(x, \tau)$ on $M_\tau\cap B_R$. 
\end{prop}
\begin{proof}
  Fix $\lambda_1, \kappa_3$ as in the statement, and
  suppose that no suitable $\lambda_2, R_2, \delta_2$ exist. Then we can
  find sequences $M^k_\tau, N^k_\tau$ of rescaled flows for $\tau\in
  [0,2]$, that are $k^{-1}$-graphical over $\mathcal{C}$ on $B_{R_k}$
  for some $R_k > k$, and which satisfy the conditions (a), (b), (c)
  with $\delta = \delta_k = k^{-1}$, but do not satisfy the desired conclusion
  with $\lambda_2 = \lambda_1 + k^{-1}$,
  i.e. we have
  \[ \Vert v_k(2)\Vert_{B_{R_k}} < e^{\lambda_1 + k^{-1}} \Vert
    v_k(1)\Vert_{B_{R_k}}. \]
  Write $N^k_\tau$ as the graph of $v_k$ over
  $M^k_\tau$ on $B_{R_k}$, and let us define
  \[ d_k = \Vert v_k(1)\Vert_{B_{R_k}}. \]
  By assumption we have $\Vert v_k(0)\Vert_{B_{R_k}}\leq
  e^{-\lambda_1}d_k$. It follows that the normalized functions
  $d_k^{-1}v_k$ converge smoothly on compact subsets of $(0,2] \times
  \mathbb{R}^3$ to a solution $v_\infty$ of the linearized equation on
  $\mathcal{C}$, and $\Vert v_\infty(2)\Vert_{\mathcal{C}}\leq
  e^{\lambda_1}$. We claim that $\limsup_{\tau\to 0} \Vert v_\infty(\tau)\Vert_{\mathcal{C}}
  \leq e^{-\lambda_1}$, and  $\Vert v_\infty(1)\Vert_{\mathcal{C}}
  = 1$. Using \cite[Theorem 0.6]{CM21} this will contradict that there are no eigenfunctions of
  $\mathcal{L}_{\mathcal{C}}$ with eigenvalue $\lambda_1$. 

  To control $\limsup_{\tau\to 0} \Vert
  v_\infty(\tau)\Vert_{\mathcal{C}}$ we bound $\Vert v_k(\tau)
  \Vert_{B_{R_k}}$ for small $\tau$. We apply the estimate
  \eqref{eq:L2growth}, with $\kappa_3/2$ instead of $\kappa_3$, to get
  \[ \Vert v_k(\tau)\Vert_{B_{R_k}} &\leq e^{c\tau/2} \Vert
    v_k(0)\Vert_{B_{R_k}} + C_{\kappa_3/2} \delta
    e^{-(R_k-C_1-1)^2/(8+\kappa_3)} \\
    &\leq e^{c\tau/2} e^{-\lambda_1}d_k + C_{\kappa_3/2} d_k
    e^{\frac{R_k^2}{8+2\kappa_3} - \frac{(R_k-C_1-1)^2}{8+\kappa_3}}.
  \]
  Once $R_k > k$ are sufficiently large (depending on $\kappa_3$), we will
  have
  \[ \Vert d_k^{-1}v_k(\tau)\Vert_{B_{R_k}}  \leq e^{c\tau/2} e^{-\lambda_1} +
    k^{-1}. \]
  Letting $k\to\infty$, and then $\tau \to 0$ we get $\limsup_{\tau\to 0} \Vert
  v_\infty(\tau)\Vert_{\mathcal{C}} \leq e^{-\lambda_1}$. 

  We next show that $\Vert v_\infty(1)\Vert_{\mathcal{C}}
  = 1$.  Let $\epsilon > 0$. We apply Proposition~\ref{prop:nonconc} to the
  $v_k$, with $R$ chosen so that $e^{1/2}(R-C_1-2) = R_k$ (which we
  can assume is greater than $R$). Then the
  conclusion, for sufficiently large $k$, is that on $B_{R_k}$ we have
  \[ |v_k(1,x)| &\leq C \Big(e^{-\lambda_1}d_k +  \delta_k
    e^{-R_k^2/2(4+\kappa_3)}\Big) e^{|x|^2/8p_0} \\
    &\leq 2Cd_k e^{|x|^2 / 8p_0}. \]
  In particular this implies that we have a bound $\Vert
  v_k(1)\Vert_{L^p(B_k)} < Cd_k$ for $2 < p < 2p_0$, for a larger
  constant $C$.  
  Using H\"older's inequality, for any $\epsilon > 0$ we can find a
  compact set $K\subset \mathbb{R}^3$ (independent of $k$) such that 
  \[ \int_{B_{R_k}\setminus K} |v_k(1,x)|^2 e^{-|x|^2/4} \leq
    \epsilon d_k2. \]
  Using the smooth convergence of $d_k^{-1}v_k(1)$ to $v_\infty(1)$ on
  compact sets, it follows from this, letting $\epsilon\to 0$ as well,
  that
  \[ \Vert v_\infty(1)\Vert_{\mathcal{C}} = \lim_{k\to \infty} \Vert
    d_k^{-1} v_k(1)\Vert_{B_{R_k}} = 1. \]
  This implies the desired contradiction. 
\end{proof}

\section{Perturbations of a degenerate cylindrical singularity}
In this section we will prove a local perturbation result. Supposing
that $L_t$ has a degenerate $\mathcal{C}$-singularity at $(0,0)$, we
will consider suitable assumptions under which a perturbed flow $L_t'$
has no $\mathcal{C}$ singularities close to the origin along the
$y$-axis. For now the size of the region where we rule out
singularities will go to zero as we consider smaller and smaller
perturbations. The result will be used in the following section along
a family of such perturbed flows to perturb away singularities in a
region of a definite size.

Suppose that $L_t$ has a degenerate $\mathcal{C}$-singularity at
$(0,0)$, and let us write $M_\tau$ for the corresponding rescaled mean curvature flow
centered at $(0,0)$. We will assume that $M_0$ is a
$\delta$-graph over $\mathcal{C}$ on $B_{\delta^{-1}}$, for small
$\delta$. From
Proposition~\ref{prop:degendecay} we 
know that the $L^2$-distance from $M_\tau$ to $\mathcal{C}$ is
bounded by $\eta_0 e^{-\left(\frac{1}{2} - \kappa_1\right)\tau}$. By
choosing $\delta$ above smaller, we can take $\eta_0, \kappa_1 > 0$ as
small as we like.
We first show the following graphicality property of
$M_\tau$ over $\mathcal{C}$. In \cite{SWX25_2} there are stronger
estimates, however they are asymptotics as $\tau\to \infty$, while in
our argument we will 
need estimates valid for all $\tau$, in terms of how close $M_0$ is to
the cylinder.
\begin{prop}\label{prop:graphicality}
  Suppose that the rescaled flow $M_\tau$ has a degenerate
  $\mathcal{C}$-singularity at infinity. 
  Let $R_0, \delta_2> 0$. For $\delta, \kappa_2 > 0$ 
  sufficiently small (depending on $R_0, \delta_2 > 0$), we
  have the following.  Let us define
  \[ \label{eq:Rtaudefn} R(\tau) = (2 +
    \kappa_2)\sqrt{\tau + R_0}, \]
  and suppose that $M_0$ is a
  $\delta$-graph over $\mathcal{C}$ on $B_{\delta^{-1}}$. Then for all
  $\tau \geq 0$ we have that $M_\tau$ is $\delta_2$-graphical over
  $\mathcal{C}$ on the ball $B_{R(\tau)+10}$. 
\end{prop}
\begin{proof}
  Note first that for any given $T_0 > 0$,
  with suitable choices of $\delta, \kappa_2$ the conclusion
  holds for $\tau\in [0,T_0]$.  We will show next that if the conclusion holds
  for times $\tau \leq T-1$ with sufficiently large $T$,
  then it also holds at $\tau \in [T,T+1]$. 
  We apply Proposition~\ref{prop:nonconc}, for the flow $M_\tau$
  translated in time  by $T-1$, viewed as a graph over
  $\mathcal{C}$. We can assume that $\delta_2 < \delta_1$ for the
  $\delta_1$ in Proposition~\ref{prop:nonconc}. Then the conclusion is
  that for $\tau\in [T, T+1]$ we have 
  \[ \label{eq:10} |\nabla^k v(x, \tau)|
    \leq C_{\kappa_3}\left( \Vert v(T-1)\Vert_{B_{R(T-1)}}^2
      + \delta_0^2 e^{-(R(T-1)-C_1-1)^2/(4+\kappa_3)}\right)^{1/2}
    e^{|x|^2/8p_0}, \]
  for $|x| < e^{1/2}(R(T-1) - C_1-2)$. 
  We choose $\kappa_3$ small enough so that $4+\kappa_3 < 4p_0$. From
  this it follows that
  \[ e^{-(R(T-1)-C_1)^2/(4+\kappa_3)} e^{R(T+1)^2 / 4p_0} \ll 1, \]
  as long as $T$ is sufficiently large, since the coefficient of $T$
  in the exponent is
  \[ -\frac{(2+\kappa_2)^2}{4+\kappa_3} + \frac{(2+\kappa_2)^2}{4p_0}
    < 0.\]
  In particular we can ensure
  that
  \[ C_{\kappa_3}^2\delta_0^2  e^{-(R(T-1)-C_1)^2/(4+\kappa_3)} e^{R(T+1)^2 / 4p_0} <
    \frac{1}{2} \delta_2^2. \]

  At the same time we have $\Vert v(T-1)\Vert \leq \eta_0
  e^{-(\frac{1}{2}-\kappa_1)(T-1)}$, so
  \[ \Vert v(T-1)\Vert_{B_{R(T-1)}}^2 e^{R(T+1)^2/4p_0} \leq \eta_0^2
  e^{-(1-2\kappa_1)(T-1) + \frac{1}{4p_0}(2+\kappa_2)^2(T+1
    + R_0)}. \]
If $\kappa_1,\kappa_2$ are sufficiently small, then
  given any $R_0,
  \delta_2$, we can arrange that once $T$ is sufficiently large, we
  have
  \[ C_{\kappa_3}^2\Vert v(T-1)\Vert_{B_{R(T-1)}}^2 e^{R(T+1)^2/4p_0} <
    \frac{1}{2}\delta_2^2. \]
  It follows from \eqref{eq:10} that for $\tau\in [T, T+1]$, the
surface $M_T$ is $\delta_2$-graphical over $\mathcal{C}$ on the ball $B_{R(\tau)}$.
\end{proof}

The following is our key local perturbation result. The main
point is that the $y$-interval in which we rule out cylindrical
singularities is of larger order than the size of the initial
perturbation. 
\begin{prop}\label{prop:rescaledperturb}
  Suppose that $M_\tau$ has a degenerate $\mathcal{C}$-singularity at infinity, and
  for sufficiently small $\delta > 0$ we have that $M_0$ is a
  $\delta$-graph over $\mathcal{C}$ on the ball
  $B_{\delta^{-1}}$. Suppose that $M'_\tau$ is another rescaled flow,
  and $M'_0$ is also a $\delta$-graph over $\mathcal{C}$ on
  $B_{\delta^{-1}}$. In addition assume that:
  \begin{itemize}
 \item As long as $R_0R(\tau)\epsilon e^\tau < \frac{1}{100}$ we have that
   $M'_\tau$ is the graph of $v(x,\tau)$ over $M_\tau$ on
   $B_{R(\tau)}$ for $R(\tau)$ defined by \eqref{eq:Rtaudefn}, and 
    $|v(x,\tau)| \leq R_0R(\tau)\epsilon e^\tau$ on $B_{R(\tau)}$. 
  \item $\Vert v(x,0)\Vert_{L^2(B_{R(0)})} \geq c_1 \epsilon$, for
    some $c_1 > 0$,
  \item For some $\kappa_4 > 0$ we have
    \[ \Vert v(x, 1)\Vert_{L^2(B_{R(1)})} \geq
      e^{\frac{1}{2}-\kappa_4} \Vert v(x,0)\Vert_{L^2(B_{R(0)})}, \]
  \end{itemize}
  If $\delta, R_0^{-1}, \kappa_4$ are sufficiently small,
  depending on $c_1$, then the
  mean curvature flow $L'_t$ corresponding to $M'_\tau$ does not have
  a cylindrical singularity at $(y, 0,0,0)$ for $|y| <
  \epsilon^{2/3}$. 
\end{prop}

\begin{proof}
  In the proof we will choose several small parameters, $1 \gg \kappa_1 \gg
  \kappa_2 \gg \kappa_3 \gg \kappa_4$. In fact we can choose
  $\kappa_i = \kappa^i$ for a sufficiently small $\kappa > 0$. The choice of $\delta, R_0$
  will depend on the choice of the $\kappa_i$.  
  Recall that we set
  \[ R(\tau) = (2 + \kappa_2) \sqrt{\tau + R_0}. \]
  We define
  \[ D(\tau) = \Vert v(x,\tau)\Vert_{B_{R(\tau)}} := \left(\int_{M_\tau
        \cap B_{R(\tau)}} |v(x, \tau)|^2 \, e^{-|x|^2/4}\,
      d\mathcal{H}^2\right)^{1/2}, \]
  so that   $D(1) \geq e^{\frac{1}{2}-\kappa_4}D(0)$.
  By assumption we know that as long as
  $\epsilon R_0 R(\tau) e^\tau< 100^{-1}$, we have $|v(x,\tau)| < \epsilon
  R_0 R(\tau) e^\tau$ on the ball $B_{R(\tau)}$.

  Recall from Proposition~\ref{prop:degendecay}
  that the $L^2$-distance of $M_\tau$ from $\mathcal{C}$ is at
  most $\eta_0 e^{-(\frac{1}{2}-\kappa_1)\tau}$, and by choosing
  $\delta$ small, we can arrange $\eta_0, \kappa_1$ to be as small as
  we like. Our goal is to show
  that $D(\tau)$ grows at a rate very close to $e^{\tau/2}$, until it
  dominates the $L^2$-distance from $M_\tau$ to $\mathcal{C}$. We
  choose a small $\kappa_4 > 0$, and we will apply
  Proposition~\ref{prop:3ann} with $\frac{1}{2} - \kappa_4 < \lambda_1
  < \lambda_2 < \frac{1}{2}$. For convenience we define $\lambda_{1.5}
  = \frac{1}{2}(\lambda_1 + \lambda_2)$.

  Let $L_0 > 0$ be chosen according to Proposition~\ref{prop:noC},
     with $\kappa = \frac{1}{4}$ (so $L_0$ does not depend on the choice of the
     $\kappa_i$).  Consider the following conditions for some $T \geq 1$ for
  $\delta_1$ chosen sufficiently small to apply
  Proposition~\ref{prop:nonconc} and Proposition~\ref{prop:3ann}. 
  \begin{itemize}
  \item[(i)$_T$] For all $\tau \leq T+1$ the surface $M'_\tau$ is
    $\delta_1$-graphical over $M_\tau$ on the ball $B_{R(\tau)}$,
  \item[(ii)$_T$] $D(T+1) \geq e^{\lambda_{1.5}} D(T)$,
    \item[(iii)$_T$] $D(T) \geq  c_1\epsilon
      e^{(\frac{1}{2}-\kappa_4)T}$.
      \item[(iv)$_T$] $D(T) \leq \eta_0^{1/2}
        e^{-(\frac{1}{2}-\kappa_1)T}$,
      \end{itemize}

      \bigskip
      \noindent{\bf Claim 1.} We claim that if $R_0^{-1}, \eta_0$ are
      chosen sufficiently small, depending on $L_0$, then (i)$_T$ and
      (iv)$_T$ imply (i)$_{T+2L_0}$. To see this first note that for
      any given $\delta_3 > 0$,  if we choose $R_0^{-1}, \eta_0$
      sufficiently small, then by an
      argument very similar to the proof of
      Proposition~\ref{prop:graphicality}, (i)$_T$ and (iv)$_T$ imply
      that $M'_\tau$ is $\delta_3$-graphical on the ball $B_{R(\tau)}$
      for $\tau\in [T+1, T+2]$. Choosing $\delta_3$ sufficiently
      small, depending on $L_0$, we can then apply
      Proposition~\ref{prop:pseudo} repeatedly to ensure that
      (i)$_{T+2L_0}$ holds.

      \bigskip
      \noindent{\bf Claim 2.} Next we claim that if (i)$_T$, (ii)$_T$
      (iii)$_T$ hold, then so do (ii)$_{T+1}$, (iii)$_{T+1}$. Since
      (iii)$_{T+1}$ is an immediate consequence of (ii)$_T$,(iii)$_T$,
      it is enough to show (ii)$_{T+1}$. For this we use
      Proposition~\ref{prop:3ann}. We use $R=R(T)$, and the flows
      $M_\tau, M_\tau'$ on the time interval $[T,T+2]$. For the
      graphicality estimate in hypothesis (a) we can use $\delta
      =\min\{\epsilon R_0R(\tau) e^{T+2}, \delta_1\}$.  Since
      Proposition~\ref{prop:3ann} deals with the $L^2$-norms on the
      fixed ball $B_R$ rather than the balls $B_{R(\tau)}$, we next
      estimate the contribution of the region $B_{R(\tau)}\setminus
      B_R$ to the $L^2$-norm of $v(x,\tau)$ for $\tau\in [T, T+2]$. By
      the estimate for $v$, this contribution can be bounded by
      \[ C_{\kappa_3} \epsilon R_0 R(T+2) e^{T+2}e^{-\frac{R(T)^2}{8 + 2\kappa_3}} \]
      for any $\kappa_3 > 0$. We claim that for suitable choices of
      $\kappa_2,\kappa_3, \kappa_4, R_0$ this is of lower order
      than the lower bound from (iii)$_T$, i.e. we claim that
      \[ C_{\kappa_3} \epsilon R_0 R(T+2) e^{T+2}e^{-\frac{R(T)^2}{8 +
            2\kappa_3}} \ll c_1\epsilon e^{(\frac{1}{2}-\kappa_4)T}. \]
      Equivalently, we need
      \[ c_1^{-1}C_{\kappa_3} R_0\sqrt{R_0+T+2} \exp\left(\Big(1 -
          \frac{(2+\kappa_2)^2}{8+2\kappa_3} -
          \frac{1}{2}+\kappa_4\Big)T + 2 -
          \frac{(2+\kappa_2)^2}{8+2\kappa_3} R_0\right) \ll 1. \]
      If we choose $\kappa_i=\kappa^i$ for sufficiently small $\kappa
      > 0$, then the
      coefficient of $T$ in the exponential is negative (the leading
      term is $-\kappa_2/2$), and then we can choose $R_0$
      sufficiently large (depending also on $c_1, C_{\kappa_3}$) to make
      the expression as small as desired. Using this, and that
      $\lambda_{1.5} > \lambda_1$, the condition
      (ii)$_T$ implies hypothesis (b) in Proposition~\ref{prop:3ann}.

      Similarly, we can verify hypothesis (c), using our choice
      $\delta \leq       \epsilon R_0 R(T+2) e^{T+2}$. By (iii)$_{T+1}$, and the
      discussion above, we have the lower bound
      \[ \Vert v(T+1)\Vert_{B_R} \geq \frac{1}{2} c_1\epsilon
        e^{(\frac{1}{2}-\kappa_4)(T+1)}, \]
      so for (c) we need to ensure that
      \[ \frac{1}{2} c_1\epsilon e^{(\frac{1}{2}-\kappa_4)(T+1)} \geq
        \epsilon R_0 R(T+2) e^{T+2} e^{-\frac{R(T)^2}{8 + 2\kappa_3}}. \]
      This follows exactly as above, for suitable choices of our
      constants.

      Therefore we can apply Proposition~\ref{prop:3ann}, and we
      deduce that $\Vert v(T+2)\Vert_{B_R} \geq e^{\lambda_2} \Vert
      v(T+1)\Vert_{B_R}$. By the discussion above we can replace the
      $L^2$-norms on $B_R$ with those on $B_{R(\tau)}$ and still
      obtain (ii)$_{T+1}$, using that $\lambda_2 >
      \lambda_{1.5}$. This completes the proof of Claim 2. 
      
      \bigskip
     Choosing $\delta$ small, we can assume that  (i)$_1$--(iv)$_1$
     all hold. We can then iterate the statements of the two Claims
     until we find $T^\epsilon$ where (i)--(iv) hold,
     but (iv)$_{T^\epsilon+L_0}$ fails.  Note that then
     (i)$_{T^\epsilon+j}$--(iii)$_{T^\epsilon+j}$  still hold for all
     $j\leq 2L_0$, and so in particular we have
     \[ \label{eq:Dgrow11} D(T^\epsilon+2L_0) \geq e^{L_0 \lambda_{1.5}}
       D(T^\epsilon+L_0). \]
     By (iii)$_{T^\epsilon}$ and (iv)$_{T^\epsilon}$ we also have
     \[ \label{eq:Tepsbound}
       e^{(1 - \kappa_4 - \kappa_1)T^\epsilon} \leq c_1^{-1}\eta_0^{1/2}
       \epsilon^{-1}. \]
     By choosing $\delta$ sufficiently small, we can arrange that
     $T^\epsilon$ is as large as we like.

     Consider the flow $M'_\tau$
     for $\tau\in [T^\epsilon, T^\epsilon+ 2L_0]$ as the graph of a
     function $V(x, \tau)$ over $\mathcal{C}$ on the balls
     $B_{R(\tau)}$. On this ball we can write $M_\tau$ as the graph of
     $u(x,\tau)$ over $\mathcal{C}$, and (viewing each function as a
     function on the cylinder using the graphicality) we have
     \[ \label{eq:Vdefn}
       V(x, \tau) = u(x,\tau) + v(x,\tau) + O( \delta_1
       v(x,\tau)). \]

     We fix a small $r_1 > 0$, and for $y_0$ satisfying
     $|e^{(T^\epsilon+2) / 2}y_0| < r_1$ we consider the
     translated flow $N_\tau = M_\tau' + e^{\tau/2}(y_0, 0,0)$. We
     will show that if our constants are chosen suitably, then because
     of \eqref{eq:Dgrow11}
     the $L^2$-distance of $N_\tau$ to $\mathcal{C}$ grows on the
     interval $[T^\epsilon+L_0, T^\epsilon+2L_0]$. Then
     Proposition~\ref{prop:noC} implies that $N_\tau$ cannot
     converge to a rotation of $\mathcal{C}$ as $\tau\to\infty$.

     Note that instead of computing the $L^2$-distance
     $\mathbf{d}_{\mathcal{C}}(N_\tau)$ of the translated flow, we can
     compute the $L^2$-distance of $M_\tau'$ from $\mathcal{C}$, but with a
     Gaussian centered at $(-e^{\tau/2}y_0, 0,0)$. Let $R_1 > 0$ be
     large, to be chosen below, and suppose that $|x_0| < r_1$ with
     $r_1 = R_1^{-2}$. For $|x| < R_1$ we have
     \[ \Big| |x-x_0|^2 - |x|^2 \Big| \leq r_1^2 + 2R_1r_1 <
       3R_1^{-1}, \]
     so the Gaussians centered at $x_0$ and $0$ are comparable:
     \[  \label{eq:Gausscompare}
       e^{ - \frac{3}{4}R_1^{-1}}   <  e^{-|x-x_0|^2/4} e^{|x|^2/4}
       < e^{\frac{3}{4} R_1^{-1}}, \text{ for } |x| < R_1.\]
     
     In order to control the contribution from outside of the
     $R_1$-ball, we use the non-concentration estimate,
     Proposition~\ref{prop:nonconc}. We apply the estimate, viewing
     $M_\tau'$ as a graph over $\mathcal{C}$ in the statement of
     the proposition, for $\tau$ in the
     intervals $[T, T+2]$, where $T\in [T^\epsilon,
     T^\epsilon+2L_0-2]$. We will use $R=R(T)$ in the proposition. 

     First, consider $T=T^\epsilon$. The $L^2$-norm of $V(x, T^\epsilon)$ on $B_R$
     is controlled by (iv)$_{T^\epsilon}$, and our $L^2$-distance estimate for
     $M_\tau$, so we have
     \[ \Vert V(x, T^\epsilon)\Vert_{L^2(B_R)} \leq C \eta_0^{1/2}
       e^{-(\frac{1}{2}-\kappa_1)T^\epsilon}. \]
     From \eqref{eq:Vdefn} we also have the pointwise bound $|V(x,\tau)| \leq 3\delta_1$
   on $B_{R(\tau)}$ if $\delta_1$ is chosen small enough, so from
     Proposition~\ref{prop:nonconc} we get that for $\tau\in
     [T^\epsilon+1, T^\epsilon+2]$
     \[ |V(x, \tau)| \leq C_{\kappa_3} \left( \eta_0^{1/2}
         e^{-(\frac{1}{2}-\kappa_1)T^\epsilon} + \delta_1
         e^{-\frac{(R(T^\epsilon)-C_1-1)^2}{8+2\kappa_3}}\right)
       e^{|x|^2/8p_0}, \]
       on $B_{e^{1/2}(R(T^\epsilon)-C_1-2)}$. To control the second
       term in the brackets in terms of the first, we need to ensure
       that for some constant $C > 0$ we have
       \[ C + \frac{(R(T^\epsilon)-C_1-1)^2}{8+2\kappa_3} - 
         \left(\frac{1}{2}-\kappa_1\right) T^\epsilon > 0. \]
     The leading order term in $T^\epsilon$ is
       \[ \left(\frac{(2+\kappa_2)^2}{8+2\kappa_3} - \frac{1}{2} +
           \kappa_1\right)T^\epsilon, \]
       and if we choose $\kappa_i=\kappa^i$ for sufficiently small
       $\kappa > 0$, then this coefficient is positive. It follows
       that we then have a constant $C_1$ (depending on the $\kappa_i$),
       such that
       \[ |V(x, \tau)| \leq C_1\eta_0^{1/2}
         e^{-(\frac{1}{2}-\kappa_1)T^\epsilon}
         e^{|x|^2/8p_0}, \]
       for $\tau\in [T^\epsilon+1, T^\epsilon+2]$ and $x\in
       B_{R(\tau)}$, once $T^\epsilon$ is sufficiently 
       large.

       This implies in particular that
       \[ \Vert V(x, T^\epsilon+1)\Vert_{L^2(B_{R(T^\epsilon+1)})}
         \leq C_2 \eta_0^{1/2}
         e^{-(\frac{1}{2}-\kappa_1)(T^\epsilon+1)}, \]
       where $C_2$ depends on the $\kappa_i$. We can repeat the same
       argument multiple times, to eventually obtain the pointwise bound
       \[ |V(x, \tau)| \leq C_3\eta_0^{1/2}
         e^{-(\frac{1}{2}-\kappa_1)T^\epsilon}
         e^{|x|^2/8p_0}, \]
       for all $\tau\in [T^\epsilon+1, T^\epsilon+2L_0]$, and $x\in B_{R(\tau)}$,
       where $C_3$ depends on the $\kappa_i$ (recall that also $L_0$
       depends on the $\kappa_i$).  

       We can now estimate the contribution
       from outside the $R_1$-ball in the $L^2$-distance of $M_\tau'$
       from $\mathcal{C}$, using the
       slightly translated Gaussian, for $\tau\in [T^\epsilon+1, T^\epsilon+2L_0]$. We have
      \[ \label{eq:21} \int_{M_\tau'\setminus B_{R_1}} \overline{\mathrm{dist}}(x)^2
        e^{-|x-x_0|^2/4} &\leq \int_{M_\tau'\cap(B_{R(\tau)} \setminus
          B_{R_1})} |V(x,\tau)|^2 e^{-|x-x_0|^2/4} \\
        &\qquad + C_{\kappa_3}
        e^{-R(\tau)^2/ (4+\kappa_3)}, \]
      and by H\"older's inequality, with $p^{-1} + q^{-1}=1$, we have
      \[ \label{eq:Hold} \int_{M_\tau'\cap(B_{R(\tau)} \setminus
          B_{R_1})} |V(x,\tau)|^2 e^{-|x-x_0|^2/4} &\leq \left( \int_{M'_\tau\cap
          B_{R(\tau)}} |V(x,\tau)|^{2p} e^{-|x-x_0|^2/4} \right)^{1/p}
      \\
      &\qquad \cdot
      \left( \int_{M'_\tau \setminus B_{R_1}}
        e^{-|x-x_0|^2/4}\right)^{1/q} 
    \]
    For any $\rho > 0$, if $|x_0|$ is sufficiently small, we have
    $|x-x_0|^2 \geq (1-\rho) |x|^2 - 1$, so by choosing $R_1$ large
    (recall that $r_1 = R_1^{-2}$ and $|x_0| < r_1$), we have 
    \[ |V(x, \tau)|^{2p} e^{-|x-x_0|^2/4} &\leq e^{1/4} |V(x,
      \tau)|^{2p} e^{-(1-\rho)|x|^2/4} \\
      &\leq C\eta_0^p e^{-p(1-2\kappa_1)T^\epsilon} e^{\frac{p |x|^2}{ 4p_0}
        - (1-\rho)\frac{|x|^2}{4}}, \]
    on $B_{R(\tau)}$.
     We can choose $p < p_0$, then $\rho$ sufficiently small and
     integrate, using also \eqref{eq:Hold}, to get
     \[ \left(\int_{M_\tau'\cap(B_{R(\tau)} \setminus
          B_{R_1})} |V(x,\tau)|^2 e^{-|x-x_0|^2/4} \right)^{1/2}\leq C \eta_0^{1/2}
        e^{-(\frac{1}{2}-\kappa_1)T^\epsilon} \left( \int_{M'_\tau \setminus B_{R_1}}
        e^{-|x-x_0|^2/4}\right)^{1/2q}. \]
     By choosing $R_1$ large, and ensuring that $T^\epsilon$ is 
     sufficiently large, this can be made negligible compared to
     $D(T^\epsilon+L_0)$, using that (iv)$_{T^\epsilon+L_0}$ fails.
     The other term in \eqref{eq:21} then is similarly
     negligible compared to $D(T^\epsilon+L_0)$ for
     $\tau=T^\epsilon+L_0$ and $\tau = T^\epsilon+2L_0$,
     since
     \[ C_{\kappa_3} e^{-R(T^\epsilon+L_0)^2 / (8+2\kappa_3)} \ll
       e^{-(\frac{1}{2}-\kappa_1)(T^\epsilon+2L_0)} \]
     for $\kappa_i = \kappa^i$ as above,  and then choosing $R_0$ large. 
     Finally note that the $L^2$-norm of $u(x,\tau)$ on $B_{R_1}$ is
     bounded by $\eta_0 e^{-(\frac{1}{2}-\kappa_1)\tau}$ (using that
     $M_\tau$ has a degenerate $\mathcal{C}$-singularity at infinity), which is
     also of lower order than $\eta_0^{1/2} e^{-(\frac{1}{2}
       -\kappa_1)(T^\epsilon+L_0)}$ if $\eta_0$ is chosen sufficiently
     small, depending on $L_0$. 

     The conclusion is that given any $c_2 > 0$, if we choose
     $\kappa_i = \kappa^i$ for sufficiently small $\kappa > 0$, and
     then choose $\delta, \eta_0,  R_0^{-1}, R_1^{-1}$
     sufficiently small (depending also on $L_0$), 
     then for $\tau = T^\epsilon+L_0, T^\epsilon+2L_0$ we have
     \[ |\mathbf{d}_{\mathcal{C}}(N_\tau) - D(\tau)| < c_2
       D(T^\epsilon+L_0). \]
     If we further choose $c_2$ small enough (depending on the
     $\lambda_i$), then we have 
     \[ \mathbf{d}_{\mathcal{C}}(N_{T^\epsilon+2L_0}) &\geq
       D(T^\epsilon+2L_0) - c_2 D(T^\epsilon + L_0) \\
       &\geq e^{L_0\lambda_{1.5}} D(T^\epsilon+L_0) - c_2 D(T^\epsilon+L_0) \\
       &\geq e^{L_0\lambda_1}(D(T^\epsilon+L_0) + c_2 D(T^\epsilon+L_0)) \\
     &\geq e^{L_0\lambda_1}
     \mathbf{d}_{\mathcal{C}}(N_{T^\epsilon+L_0}) \geq e^{L_0/4}
     \mathbf{d}_{\mathcal{C}}(N_{T^\epsilon+L_0}). \]
     Using Proposition~\ref{prop:noC}, it follows that $N_\tau$ cannot
    have a cylindrical singularity at infinity (with limit
    $Q\mathcal{C}$ for any rotation $Q$).

     In terms of the flow $L'_t$, this means that $L'_t$ cannot have a
    cylindrical singularity at $(y_0,0,0,0)$ with
     $|e^{(T^\epsilon+2)/2}y_0| < r_1$, i.e. with
     \[ |y_0| < r_1 e^{-(T^\epsilon+2)/2}. \]
     From \eqref{eq:Tepsbound} we know that once $\epsilon$ is
     sufficiently small, then we have $r_1 e^{-(T^\epsilon+2)/2} \geq
     \epsilon^{2/3}$. This completes the proof. 
   \end{proof}

   \section{Proof of the main results}
 In this section we will prove Theorems~\ref{thm:main} and \ref{thm:surgery}, first proving a
 result that locally perturbs away degenerate singularities. Recall that we are considering a mean 
 curvature flow $S_t \subset \mathbb{R}^3$ 
of compact surfaces which encounters only spherical
and cylindrical singularities. Recall that we can define $S_t$ for all $t \geq 0$
as a weak flow, for instance as a level set flow or Brakke flow. By
Choi-Haslhofer-Hershkovits~\cite{CHH22} and
Hershkovits-White~\cite{HW19} the different weak formulations of the
flow coincide in our setting, and are non-fattening. 

In the arguments later, we will need to consider small perturbations of
the flow $S_t$ on different time intervals $[a,b]$. When we say that a certain
result holds for flows $S_t'$ on $[a,b]$ that are sufficiently
close to $S_t$, this means that if $S^i_t$ is a sequence of unit
regular, cyclic, integral Brakke flows converging weakly to $S_t$ on $[a,b]$,
then the result holds for sufficiently large $i$. Note that this could
be quantified by metrizing the
weak convergence of Brakke flows. Below we will also have a further
perturbation $S''_t$ of $S'_t$, and we will use barriers to control
$S''_t$ as a graph over $S'_t$. 

As a preliminary step we will build suitable barrier flows in order to
control the behavior of flows close to $S_t$. 
If the flow $S_t$ were assumed to be mean convex, then we could
consider time translated solutions $S_{t\pm \epsilon}$, which could
serve as barriers on the two sides of $S_t$. In general we will use
the fact that by the work of Choi-Haslhofer-Hershkovits~\cite{CHH22}
the flow is mean convex, for suitable
orientations, near any cylindrical singularity. Away from the
cylindrical singularities the flow is smooth, so we can construct
barriers as graphs, and we will combine these with suitable time translations
near the singularities. This is related to the
method used by Hershkovits-White~\cite{HW19} to prove uniqueness of
the flow through mean convex singularities. 

Given a singularity  $X_0 := (x_0, t_0)$ of $S_t$, there is
an $r_{X_0} > 0$ such that the flow $S_t$ is mean convex in the
parabolic neighborhood
\[ Q_{r_{X_0}}(X_0) := \{(x,t)\, :\, |x - x_0| < r_{X_0}, |t - t_0| <
  r_{X_0}^2\},\]
for a suitable orientation. 
Letting $\Sigma$ be the singular set of the flow $S_t$ in space-time, we
can cover $\Sigma$ by finitely many such neighborhoods. It follows
that there are (bounded) open sets $\mathcal{U}_1
\subset\subset \mathcal{U}_2\subset
\mathbb{R}^3\times[0,\infty)$ such that if $X\in \mathcal{U}_2$,
then the flow is mean convex near $X$ for a suitable orientation,
while the flow is smooth on the
complement of $\mathcal{U}_1$. In particular we have a constant $A_0$
such that the second fundamental form satisfies $|\nabla^k A| < A_0$
on the complement of $\mathcal{U}_1$, for $k=0,1,2$.
By the strong maximum principle we can also
assume that the mean curvature is strictly positive on
$\mathcal{U}_2\setminus \mathcal{U}_1$ and so we have a lower bound
$|H| > h_0 > 0$. All of these properties also hold for mean curvature
flows $S'_t$ that are sufficiently close to $S_t$. 

If $S'_t$ is a mean curvature flow that is sufficiently close to
$S_t$, then we will construct barriers for mean curvature flows
$S''_t$ whose initial condition $S''_{t_0}$ is very close to
$S'_{t_0}$. On $\mathcal{U}_1$ these barriers will be given by
suitable time translations of $S'_t$,
while outside of $\mathcal{U}_2$ they are defined as graphs over
$S'_t$. If $U\subset \mathbb{R}^3$ is an open set, we will say that on
$U$ a hypersurface $S''$ lies 
between $S'_{t_1\pm \epsilon}$ for some $t_1, \epsilon$, if the following
holds: we have $\mathcal{U}_1 \subset U\times [t_1-\epsilon,
    t_1+\epsilon] \subset \mathcal{U}_2$, so in particular on each connected 
component of $U$ the mean curvature of $S'_t$ has a definite sign for
$t\in [t_1-\epsilon, t_1+\epsilon]$ and we can write $S'_t\cap
U = \partial \Omega_t\cap U$ for sets $\Omega_t$
with $\Omega_t \subset \Omega_{t'}$ for $t < t'$. We then require that
$S''\cap U \subset \Omega_{t_1 + \epsilon}$ while $S'' \cap
\Omega_{t_1 -\epsilon} = \emptyset$. 

The main result that we need is the following.

\begin{prop}\label{prop:barrier}
  There are $\gamma, \sigma > 0$ depending on the flow $S_t$, such
  that for all $T > 0$ we have the following. We can choose open sets
  $U_1 \subset\subset U_2\subset \mathbb{R}^3$ such that on the time interval
  $I := [T-2\gamma, T+2\gamma]$ we have
  \[\label{eq:Ucontain} \mathcal{U}_1 \subset U_1 \times I \subset U_2 \times I\subset
    \mathcal{U}_2. \]
  Suppose that $t_0\in
  [T-\gamma, T+\gamma]$, $t_0' < t_0$, and $S'_t$ is a mean curvature
  flow on $[t_0', T+2\gamma]$ that is
  sufficiently close to $S_t$. Suppose in addition that $S''_t$ is
  a mean curvature flow on $[t_0, T+2\gamma]$ satisfying that for some
  $\epsilon < \min\{\sigma, t_0-t_0'\}$ we have
  \begin{enumerate}
    \item Outside of $U_1$, the surface $S''_{t_0}$ is in the
      $\sigma\epsilon$-neighborhood of $S'_{t_0}$,
    \item In $U_2$, the surface $S''_{t_0}$ lies between $S'_{t_0 \pm
        \epsilon}$.
    \end{enumerate}
  Then for $t\in [t_0, T+\gamma]$, on $U_1$, the surface $S''_t$ lies
  between $S'_{t\pm \epsilon}$. 
\end{prop}
\begin{proof}
  We will first assume that $S'_t = S_t$.
  We first fix $ T> 0$, and then at the end we show that the constants
  that we obtained can be chosen independently of $T$.
  To construct the sets $U_1 \subset\subset U_2$,
 note that in the $t=T$ time slice we have $\overline{\mathcal{U}_1}\cap
\{t=T\} \subset \mathcal{U}_2\cap \{t=T\}$. We can choose
$U_1\subset\subset U_2$ to be nested between these sets.
 Then for
sufficiently small $\gamma$ we will have
\eqref{eq:Ucontain}. Perturbing the sets slightly, and decreasing
$\gamma$ if necessary, we can assume that $\partial U_1, \partial U_2$
are smooth, and they intersect $S_t$ transversely for $t\in
I = [T-2\gamma, T+2\gamma]$. 
We also
choose sets $U_1^{\pm}$ and $U_2^\pm$ that satisfy the same
properties, and so that we have
\[ U_1^- \subset\subset U_1 \subset\subset U_1^+\subset\subset U_2^-
  \subset\subset U_2 \subset\subset U_2^+. \]
 Note that $S_t$ has uniform curvature bounds
outside of $U_1^-$ for $t\in I$. Further shrinking $\gamma$, we can
assume that we can parametrize a subset of $S_t\cap (U_2^+\setminus U_1^-)$
using normal graphs of $u(x,t)$ defined on $S_T\cap (U_2\setminus
U_1)$. Since the mean
curvature is strictly bounded away from zero, we have a constant $C_1$
such that for any $t_1 < t_2$ we have 
\[ C_1^{-1} (t_2 - t_1) < |u(x,t_2) - u(x,t_1)| < C_1 (t_2-t_1). \]
For small $\epsilon > 0$ we define the functions $u^\epsilon(x,t) = u(x, t\pm \epsilon)$
on $S_T\cap (U_2 \setminus U_1)$, where the sign of $\epsilon$ is
chosen on each connected component to make $u^\epsilon > u$. It follows that we have
\[ C_1^{-1} \epsilon < u^\epsilon(x,t) - u(x,t) < C_1\epsilon. \]
Note that $u^\epsilon(x,t)$ parametrizes a subset of $S_{t\pm \epsilon}\cap
(U_2^+\setminus U_1^-)$.

Shrinking $\gamma$ even more, we can also parametrize a subset of 
$S_t\setminus U_1^-$ using normal graphs of $v(x,t)$ over  $S_T\setminus
U_1$ (note that $v=u$ on $U_2\setminus U_1$).
We will construct a positive supersolution of the graphical mean curvature
flow over $S_T\setminus U_1$ for $t\geq t_0$. Let $V >
0$ be a smooth function on 
$S_T\setminus U_1$ such that $V < C_1^{-1}/2$ outside of $U_2^-$, and
$V > 2C_1$ in $U_1^+$. We can choose such $V$ to be the
restriction of a smooth function from $\mathbb{R}^3$. Using that
$U_1^+\subset\subset U_2^-$, we can arrange for
$V$ to have derivatives bounded uniformly in terms of the distance
between $U_1^+$ and $\partial U_2^-$, which in turn can be  controlled
using the distance between $\mathcal{U}_1$ and $\partial\mathcal{U}_2$. 
For $\epsilon > 0$ let us define
\[ v^\epsilon(x,t) = v(x,t) + \epsilon( V(x) + K(t-t_0)), \]
for large $K > 0$ to be chosen. Since $v(x,t)$ satisfies the graphical mean
curvature flow, and using that outside of $U_1^-$ the curvature is
uniformly bounded, it follows that if $K$ is sufficiently large and
$\epsilon$ sufficiently small, depending on the curvature bounds and
the bounds for the derivatives of $V$, then $v^\epsilon(x,t)$ is a
supersolution of the graphical mean curvature flow equation
over $S_T\setminus U_1$ for
$t\geq t_0$.

We define a global supersolution $N^\epsilon_t$ of the mean curvature flow for $t\in
[t_0, T+2\gamma]$ as follows. Inside $U_1^+$ we set $N^\epsilon_t =
S_{t\pm \epsilon}$, where the sign is chosen as above. Outside of
$U_2^-$, we let $N^\epsilon_t$ be the graph of
$v^\epsilon(x,t)$. In the 
annular region $U_2\setminus U_1$, we let
$N^\epsilon_t$ be the graph 
of
\[ V^\epsilon(x,t) := \min\{u^\epsilon(x,t), v^\epsilon(x,t)\}.\]
We claim that if
$t-t_0$ is sufficiently small (i.e. $\gamma$ is small), then this is
well defined. Indeed, on $U_2\setminus U_2^-$ we have $u^\epsilon(x,t) >
u(x,t) + C_1^{-1}\epsilon$, but at the same time
\[ v^\epsilon(x,t) < u(x,t) + \epsilon (C_1^{-1}/2 + K(t-t_0)). \]
So if $K(t-t_0) < C_1^{-1}/2$, then near $\partial U_2$ the minimum is
given by $v^\epsilon$. In $U_1^+\setminus U_1$ we have $u^\epsilon(x,t) <
u(x,t) + C_1\epsilon$, but also
\[ v^\epsilon(x,t) > u(x,t) + 2\epsilon C_1, \]
so here the minimum is given by $u^\epsilon(x,t)$. Moreover the graph
of $u^\epsilon(x,t)$ agrees with $S_{t\pm \epsilon}$ in $U_2$, so the
three pieces fit together to give a supersolution outside of $U_1$ in
the barrier sense (note that $V^\epsilon$ is not smooth), and a
solution of the mean curvature flow inside
$U_1$. Note that outside of $U_1$ we have $V^\epsilon(x,t) > v(x,t) +
\sigma\epsilon$, where $\sigma = \min\{C_1^{-1}, \inf V\}$. 

We similarly define $N^\epsilon_t$ for negative $\epsilon$ by
switching the orientation. 
We now consider a flow $S''_t$ satisfying the conditions (1), (2) in
the proposition. These conditions, together with the lower bound for
$V^\epsilon$ and the fact that on $U_1$ the $N^\epsilon_t$ coincide
with $S_{t\pm \epsilon}$, imply that $S''_{t_0}$ lies between $N^{\pm
  \epsilon}_{t_0}$. Using the supersolution property (for different
orientations depending on the sign of $\epsilon$) it follows that
$S''_t$ lies between $N^{\pm \epsilon}_t$ for $t \geq t_0$ as
well. The conclusion follows since on $U_1$ we have $N^{\pm
  \epsilon}_t = S_{t\pm \epsilon}$.

So far we have assumed that $S'_t = S_t$. In general, if $S'_t$ is
sufficiently close to $S_t$, then $S'_t$ satisfies the same curvature
bounds outside of $U_1$ as $S_t$, and so the same barrier construction
will work, with the same constants. To see that the constants $\gamma,
\sigma$ can be chosen independently of $T$, we can argue as follows.
Writing $T_{ext}$ for the extinction time of $S_t$, we can cover $[0,
T_{ext}+1]$ by finitely many intervals of the form $(T_i - \gamma_i/2,
T_i + \gamma_i/2)$. We can then let $\gamma, \sigma$ be the smallest
of the corresponding $\gamma_i/2$ and $\sigma_i$.
\end{proof}

In our application we will consider a flow $S'_t$ that is sufficiently close
to $S_t$ on an interval $[t_1', T_{ext}+1]$ so that
Proposition~\ref{prop:barrier} can be applied. We will assume moreover
that $S'_t$ is smooth for $t\in [t_1, t_2]$ for some $t_2 \geq t_1$
and $t_2 > t_1'$, and has a degenerate
cylindrical singularity at a point $(x_3,t_3)$ with $t_3 > t_2$ very close to
$t_2$, so that $t_3-t_2 < \gamma$. We will construct suitable
perturbations $S''_{t_2}$ of $S'_{t_2}$, and then use
Proposition~\ref{prop:barrier} in order to control the corresponding
flow $S''_t$ at times $t\in [t_2, t_3]$. The perturbation 
$S''_{t_2}$ will in turn be obtained by constructing a perturbation
$S''_{t_1}$ of $S'_{t_1}$. In the application to Theorem~\ref{thm:main} we will have
$t_1=0$, while for Theorem~\ref{thm:surgery} we will use $t_1=t_2$.

We let $L_t$ be a suitable
rotation, translation and scaling of the flow $S'_t$, so that $L_t$ has a
degenerate $\mathcal{C}$-singularity at $(0,0)$ corresponding to the
singularity $(x_3,t_3)$ of $S'_t$. For a small $\delta_0 >
0$ we will assume that $L_{-2}$ is $\delta_0$-graphical over $\sqrt{2}\mathcal{C}$ on
 $B_{\delta_0^{-1}}$, and moreover $L_{-1}$ corresponds to a
 translation, rotation and scaling of $S'_{t_2}$. Let us denote by $M_\tau$ the corresponding
 rescaled flow with $M_0=L_{-1}$. Below we will choose the $R_0$ in the definition
 \eqref{eq:Rtaudefn}  of $R(\tau)$ very large. Using the graphicality
 from Proposition~\ref{prop:graphicality}, taking $\delta_0$
 sufficiently small, we can arrange that after scaling back to
 the flow $S'_t$, the balls $M_\tau\cap B_{R(\tau)+10}$ are
 contained in the set $U_1$ that appears in
 Proposition~\ref{prop:barrier}, applied for the interval
 $[t_2,t_3]$. 
 
In order to perturb away the degenerate singularity of $L_t$, we
want to consider the flows $L^a_t$, where $L^a_{-1}$ is roughly the
graph of $a\chi_{R_0}y$ over
$L_{-1}$ for
$|a|$ sufficiently small. Here $\chi_{R_0}$ is a cutoff function
supported in $B_{R_0}$ and equal to 1 on $B_{R_0-1}$, and $y$ is the
coordinate along the $\mathbb{R}$-factor of $\mathcal{C}$. When
$t_1 < t_2$, then in order to
achieve this using a perturbation of the initial condition $S'_{t_1}$, we will
need the following result from \cite[Lemma 5.3]{SX24} (see also
\cite[Lemma 5.2]{HV92}).
\begin{prop}\label{prop:pert1}
  For $t\in [t_1, t_2]$, write $\mathcal{P}_t:L^2(S'_{t_1})\to L^2(S'_t)$
  for the map satisfying the linearized equation $\partial_t P_tf =
  \Delta \mathcal{P}_tf + |A|^2\mathcal{P}_t f$, and $\mathcal{P}_{t_1}f=f$ for all $f$. Then for
  any $t \in [t_1, t_2]$ the image of $\mathcal{P}_t$ is dense in
  $W^{1,p}(S'_t)$, for any $p > 1$. 
\end{prop}

\begin{remark}\label{rem:initialperturb}
  We expect that even if the flow has some nondegenerate cylindrical
  singularities for $t\in (t_1,t_2)$, the image of $\mathcal{P}_{t}$
  is dense whenever $S'_{t}$ is smooth. If such an extension of
  Proposition~\ref{prop:pert1} holds, then it seems likely
  that the statement of
  Theorem~\ref{thm:main} can be extended for all time, not just up to the
  first singular time.
\end{remark}

We will use $p > 2$ so that $W^{1,p}\subset L^\infty$, and note that
the image of $C^2$ under $\mathcal{P}_t$ is also dense. Let $\omega >
0$ be small, to be chosen. As above, suppose that
$L_{-1}$ is a suitable translation, rotation, and
scaling of $S'_{t_2}$, and to be specific suppose that
\[ \label{eq:71} L_{-1} = \Lambda_0 Q_0S'_{t_2} + x_3, \]
for $\Lambda_0\in \mathbb{R}, Q_0\in SO(3), x_3 \in \mathbb{R}^3$.
We can view the function $f = \chi_{R_0} y$
(defined initially on $L_{-1}$) as a function $\tilde{f}$ on
$S'_{t_2}$, defined by $\tilde{f}(x) = f(\Lambda_0 Q_0 x + x_3)$.
By Proposition~\ref{prop:pert1} we can find a $C^2$ function $h$ on $S'_{t_1}$ so
that
\[ \label{eq:hdefn1}
  |\mathcal{P}_{t_2} h - \tilde{f}| < \omega.\]
For small $a$, we then let
$S'^a_t$ denote the mean curvature flows such that the initial surface $S'^a_{t_1}$ is the
graph of $ah$ over $S'_{t_1}$. Let us denote by $L^a_{-1}$ the surfaces
obtained from $S'^a_{t_2}$ using the same transformation as in
\eqref{eq:71}, and let $L^a_t$ denote the corresponding flows, which
can be defined for $t \geq -2$.  We also let $M^a_\tau$ denote the rescaled mean curvature
flows, defined for $\tau\geq -\ln 2$,  for which $M^a_0=L^a_{-1}$. 

We will use Proposition~\ref{prop:barrier} to compare
the flows $L^a_t$ for different values of $a$. For this we will need the
following, 
\begin{prop}\label{prop:approxy}
  Let $\epsilon > 0$, and suppose that $|a|, |a'|, \omega$ are sufficiently small,
  depending on $\epsilon$ and on the flow $S'_t$ for $t\in [t_1,t_2]$. Then
  $S'^{a'}_{t_2}$ is the graph of a function $u$ over $S'^a_{t_2}$,
  where
  \[ \label{eq:72} |u - (a'-a)\tilde{f}| \leq \epsilon |a'-a|. \]
  Here we are viewing $\tilde{f}$ as a function on $S'^a_{t_2}$ by
  viewing $S'^a_{t_2}$ as a graph over $S'_{t_2}$ and using the same
  scaling as in \eqref{eq:71}. 
\end{prop}
\begin{proof}
  To see this, consider first the linearized operator along the flow
  $S'^a_t$. Write $\mathcal{P}^a_t$ for its solution operator with a
  given initial condition at time $t_1$. As long as $a$ is sufficiently small, we will have that
  $|\mathcal{P}^a_th - \mathcal{P}_th| < \frac{1}{2}\epsilon$ for
  $t\leq t_2$, viewing $\mathcal{P}_th$ as a 
  function on $S'^a_t$ using the graphicality over $S'_t$. Next, note that
  the surface $S'^{a'}_{t_1}$ is the graph of $(a'-a)\tilde{h}$ over
  $S'^a_{t_1}$, where
  \[ \Big| (a'-a)\tilde{h} - (a'-a)h \Big| \leq C(|a| +
    |a'|)(|a'-a|). \]
  The flow with initial condition $S'^{a'}_{t_1}$ can be approximated by
  the solution $(a'-a)\mathcal{P}^a_t\tilde{h}$ of the corresponding
  linear equation, and the error at $t=t_2$ will be quadratic in
  $|a'-a|$. Therefore  $S'^{a'}_{t_2}$
  is the graph of $u$ over $S'^a_{t_2}$, where
  \[ |u - (a'-a)\mathcal{P}^a_{t_2}\tilde{h}| \leq C |a'-a|^2. \]
  Using the estimates above, we then have
  \[ |u - (a'-a)\tilde{f}| &\leq |u - (a'-a)\mathcal{P}_{t_2} h| + |a'-a|
    |\mathcal{P}_{t_2} h - \tilde{f}| \\
    &\leq C|a'-a|^2 + \frac{1}{2}\epsilon
    |a'-a| + \omega |a'-a|. \]
  If $|a|, |a'|, \omega$ are sufficiently small (depending on
  $\epsilon$) then we get the required result. 
\end{proof}

\begin{remark}\label{rem:barrier}
Let us describe here how this result will be used in conjunction with
Proposition~\ref{prop:barrier}, and in particular how small we need to
choose $\epsilon$ (and in turn $\omega, a,a'$ by
Proposition~\ref{prop:approxy}).  The actual application will
be slightly more complicated, involving  translations of
these flows in spacetime. Consider the rescaled flows $M^a_\tau$ and
$M^{a'}_\tau$. The estimate \eqref{eq:72} says that we can write
$M^{a'}_0$ as the graph of a function $u$ over $M^a_0$, where
\[ \label{eq:73} |u - (a'-a)\chi_{R_0}y| \leq \Lambda_0 \epsilon |a'-a|, \]
and $\Lambda_0 > 0$ is the scaling factor from \eqref{eq:71}. In this
argument we are treating $\Lambda_0$ as fixed, so we can choose
$\epsilon$ small to make $\Lambda_0\epsilon$ as small as we like. We
can also arrange that along the rescaled flow $M^a_\tau$ the balls
$B_{R(\tau)+10}$ (when scaled back to the original flow)
are contained in the set $U_1$ used in
Proposition~\ref{prop:barrier}. For given $s$, the time translated flow $L^a_{t+s}$ 
corresponds to the rescaled flow
\[ \label{eq:Mas} M^{a,s}_\tau = (1 - e^\tau s)^{1/2} M^a_{\tau - \log (1 - e^\tau s)}. \]
If $a$ is sufficiently small, then for
$\tau\in [-1/2, 1/2]$ the flow $M^a_\tau$ is still
$\delta_0^{1/2}$-graphical over $\mathcal{C}$ on
$B_{\delta_0^{-1/2}}$, and so for sufficiently small $s$ the surface
$M^{a,s}_0$ is the graph of a function $v_s$ over $M^a_0$ with $v_s$
having the opposite sign of $s$, and $|v_s| > s/2$ on
$B_{\delta_0^{-1/2}}$ (actually $v_s \sim -\frac{\sqrt{2}}{2} s$).

If we choose $s=2|a'-a|R_0$, then \eqref{eq:73} implies that on
$B_{R_0}$, the surface $M^{a'}_0$ lies between $M^{a,\pm
  s}_0$, while outside of $B_{R_0}$ it is of distance at most
$\Lambda_0\epsilon |a'-a|$ from $M^a_0$. Rephrasing this in terms of
the original flows, we have that on a certain ball $\tilde{B}_{\Lambda_0^{-1}R_0}$,
the surface $S'^{a'}_{t_2}$ is contained between $S'^a_{t_2\pm
  \Lambda_0^{-2}s}$, while outside of $\tilde{B}_{\Lambda_0^{-1}R_0}$
the surface $S'^{a'}_{t_2}$ is contained in the $\epsilon
|a'-a|$-neighborhood of $S'^a_{t_2}$. Note that a subset of the
complement of $\tilde{B}_{\Lambda_0^{-1}R_0}$ will generally be
contained in the region $U_1$. However, since the mean
curvature is strictly positive there, it still follows that if we
choose $\epsilon$ sufficiently small, then $S'^{a'}_{t_2}$ will be
contained between $S'^a_{t_2\pm \Lambda_0^{-2}s}$ on all of
$U_1$. Choosing $\epsilon$ even smaller if necessary we can
also ensure that $S'^{a'}_{t_2}$ is in the $\sigma |a'-a|$-neighborhood
of $S'^{a}_{t_2}$ outside of
$\mathcal{U}_1$. Applying Proposition~\ref{prop:barrier} we find that
for $t\in [t_2, t_3]$ the surface $S'^{a'}_t$ lies
between $S'^a_{t\pm \Lambda_0^{-2}s}$ on $U_1$.

Scaling this back to the rescaled flows, we find that on the balls
$B_{R(\tau)}$ the surface $M^{a'}_\tau$ lies between the surfaces
$M^{a, \pm s}_\tau$. As long as
$e^\tau s = R_0R(\tau)|a'-a|e^\tau$ is sufficiently small, by
Lemma~\ref{lem:geomJacobi2} we then find that $M^{a'}_\tau$ is the
graph of a function $v$ over $M^a_\tau$ on the ball $B_{R(\tau)+10}$,
with $|v| < CR_0R(\tau)|a'-a|e^\tau$, for a uniform constant $C$. This
estimate can then be used when applying
Proposition~\ref{prop:rescaledperturb}. 
\end{remark}

We will next show, roughly speaking, that
if $L^a_t$ has a degenerate cylindrical 
singularity at some point with $y$-coordinate $y_0$, then for $a'$
close to $a$ the flow $L^{a'}_t$ cannot 
have a cylindrical singularity with $y$-coordinate $y_0'$, if
$|y_0-y_0'| < |a-a'|^{3/4}$. For this, we will also consider small
 translations in space and time of the
 flows $L^a_t$. Let $\alpha\in \mathbb{R}$ and $\Phi =
 (\beta_1, \beta_2, s)$, where 
 $\beta_1,\beta_2, s\in \mathbb{R}$. We write $|\Phi| < \epsilon$ if
 \[ |\beta_1| + |\beta_2| + |s|^{1/2} < \epsilon. \]
 We define the transformed flow
 \[ L^{a, \alpha, \Phi}_t = L^a_{t+s} + (\alpha, \beta_1,
   \beta_2). \]
 Note that translation in the $y$-direction plays a special
 role since $\mathcal{C}$ is invariant under it, so we deal with it
 separately with the parameter $\alpha$ in the notation. 
 The key result is the following. Recall here that, as in the
 discussion after Proposition~\ref{prop:barrier}, $L_t$ is a suitable
 translation, rotation and scaling of the flow $S'_t$, which in turn
 is a very small perturbation of $S_t$. The perturbation $L^a_t$ is defined by
 taking the graph of $ah$ over $S'_{t_1}$, where $h$ is as in
 \eqref{eq:hdefn1}. 
 \begin{prop}\label{prop:localperturb}
   Suppose that $R_0$ is sufficiently large, and $\delta_0$ as well as
   $\omega$ in the definition of $h$ is sufficiently small. Assume that $L_{-2}$ is
   $\delta_0$-graphical over $\sqrt{2}\mathcal{C}$ on
   $B_{\delta_0^{-1}}$. There is an $\epsilon_0 > 0$ depending on the
   original flow $S_t$, and an $\epsilon_1 < \epsilon_0$ depending on
   the choice of 
   $L_t$ and $h$ such that we have the following: suppose that $a, a', 
   \alpha, \alpha', \Phi, \Phi'$ are such that
   \[ |a|, |a'| < \epsilon_1, \text{ and }\,
     |\alpha| , |\alpha'|, |\Phi|, |\Phi'| < \epsilon_0, \]
   and the flows $L^{a,\alpha, \Phi}_t$ and $L_t^{a', \alpha',
     \Phi'}$ have degenerate cylindrical singularities at
   $(0,0)$. Then $|\alpha - \alpha'| \geq |a-a'|^{3/4}$. 
 \end{prop}
 \begin{proof}
   Applying a rotation $Q$ close to the identity, we can assume that $QL_t^{a,\alpha, \Phi}$ has
   a degenerate $\mathcal{C}$-singularity at $(0,0)$, while
   $QL_t^{a',\alpha', \Phi'}$ has a cylindrical singularity (modeled on
   a rotation of   $\mathcal{C}$) at $(0,0)$. 
   Let $M_\tau$ and $M'_{\tau}$ be the corresponding rescaled flows, centered at
   $(0,0)$, corresponding to $QL^{a,\alpha, \Phi}_t$ and $QL^{a',
     \alpha, \Phi'}_t$ (note that we are using $\alpha$ in both). 
   For any $\delta_1 > 0$, if  $\delta_0, \epsilon_0$ are small enough, then
   we can assume that both $M_\tau$ and $M_\tau'$ are
   still $\delta_1$-graphical over $\mathcal{C}$ on
   $B_{\delta_1^{-1}}$, for $\tau\in [-\frac{1}{2}, 1]$.  Writing
   $\gamma := |a-a'| + |\Phi - \Phi'|$, our goal   is to show that
   the flow $L^{a', \alpha, \Phi'}_t$  cannot have a cylindrical
   singularity at $(y,0,0,0)$ with $|y| <
   \gamma^{3/4}$. For this we will use
   Proposition~\ref{prop:rescaledperturb}.

   We are assuming that $M_\tau$ has a degenerate cylindrical singularity at
   infinity, so Proposition~\ref{prop:graphicality} applies. For some
   range of $\tau$, we can write $M'_\tau$ as the graph of $v(x,\tau)$ over $M_\tau$ on
   $B_{R(\tau)}$, where $R(\tau)$ is defined as in
   \eqref{eq:Rtaudefn}, for $\tau\in [-\frac{1}{2}, \infty)$.
   We claim the following properties:

   \begin{itemize}
     \setlength{\itemsep}{10pt}
   \item[{\bf Claim 1:}] For a uniform constant $C_3 > 0$ we have
   $\Vert v(x,0)\Vert_{L^2(B_{R(0)})} \geq C_3^{-1}\gamma$, if
 $\epsilon_0, \delta_0, \omega, R_0^{-1}$ are small enough.  To see this, note that by
 Proposition~\ref{prop:approxy} and 
 Lemma~\ref{lem:geomJacobi}, to leading order
 on the ball $B_{R(0)/2}$, we can write the surface $L^{a', \alpha,
   \Phi'}_{-1}$ as the graph of
 \[ \label{eq:51} (a'-a)  y + F_{\Phi' - \Phi} \]
 over $L^{a, \alpha, \Phi}_{-1}$, where $F_{\Phi' - \Phi}$ is a
 combination of the eigenfunctions $1, z_1, z_2$ of
 $\mathcal{L}_{\mathcal{C}}$ 
 corresponding to translations in time and space (in the
 $z_1,z_2$-directions), given by the difference
 of $\Phi'$ and $\Phi$. Since these eigenfunctions are mutually
 $L^2$-orthogonal, and also orthogonal to the function $y$, we get
 the lower bound for $\Vert v(x,0)\Vert_{L^2(B_{R(0)})} \geq C_3^{-1}\gamma$.
 
 \item[{\bf Claim 2:}] Given any $\kappa > 0$, once $\epsilon_0, \delta_0, R_0^{-1}$ are
   small enough, we have
   \[ \Vert v(x, 1)\Vert_{L^2(B_{R(1)})} \geq e^{\frac{1}{2} -
       \kappa} \Vert v(x,0)\Vert_{L^2(B_{R(0)})}. \]
   This also follows from the discussion above, showing that to
   leading order $v(x,0)$ is given by a combination of eigenfunctions
   as in \eqref{eq:51}. We obtain the claimed growth estimate using
   that these eigenfunctions 
   of $\mathcal{L}_{\mathcal{C}}$ all have eigenvalues at least
   $1/2$ and are mutually $L^2$-orthogonal. 
   
   \item[{\bf Claim 3:}] If $\epsilon_1$ is also sufficiently small,
     depending on $L_t$, then for a uniform constant $C > 0$ we have
   that $|v(x,\tau)| \leq CR_0R(\tau)\gamma e^\tau$, as long as $CR_0R(\tau)\gamma
   e^\tau < 100^{-1}$. To see this, we first use
   Proposition~\ref{prop:barrier} following the discussion in
   Remark~\ref{rem:barrier}
   to compare the rescaled flow
   $N_\tau'$ corresponding to $L^{a',\alpha,\Phi}_t$ to the
   rescaled flow $M_\tau$ corresponding to $L^{a,\alpha,\Phi}_t$
   (i.e. we only change $a$ to $a'$). Let us write
   \[ \Phi &= (\beta_1,\beta_2,s), \\
       \Phi' &= (\beta_1', \beta_2', s'). \]
     As in Remark~\ref{rem:barrier}, we have that if $\epsilon_1$ is
     sufficiently small (depending on $L_t$), then $L^{a'}_{-1}$ is the
     graph of a function $u$ over $L^a_{-1}$, satisfying
     \eqref{eq:73}. 
     It follows that if $\gamma$ is sufficiently small, then
     $L^{a', \alpha, \Phi}_{-1-s}$ is the graph of a
     function $\tilde{u}$ over $L^{a,\alpha, \Phi}_{-1-s}$ satisfying
     \[ |\tilde{u} - (a'-a)\chi_{R_0}y| \leq 2\Lambda_0 \epsilon
       \gamma. \]
     We can now argue exactly as in Remark~\ref{rem:barrier} to deduce
     that as long as $R_0\gamma e^\tau$ is sufficiently small,
     $N'_\tau$ is the graph of $\tilde{u}$ over $M_\tau$ on
     $B_{R(\tau)+6}$, with $|\tilde{u}| < CR_0R(\tau)\gamma e^\tau$. 

   Let us now relate $M'_\tau$ to $N'_\tau$. For this note that
   $M'_\tau$ is the rescaled flow corresponding to
   \[ L^{a',\alpha,\Phi'}_t = L^{a'}_{t+s'} + (\alpha, \beta_1',
     \beta_2'), \]
   while $N'_\tau$ corresponds to
   \[ L^{a', \alpha,\Phi}_t = L^{a'}_{t+s} + (\alpha,
     \beta_1,\beta_2). \]
   We have
   \[ L^{a',\alpha, \Phi'}_t =  L^{a', \alpha,
       \Phi}_{t+(s'-s)} + (0, \beta_1'-\beta_1, \beta_2'-\beta_2), \]
   and so if we write $x_0 = (0, \beta_1'-\beta_1, \beta_2'-\beta_2)$, we have
   \[ M'_\tau  = (1 - e^\tau(s'-s))^{1/2} N'_{\tau -
       \log(1-e^\tau(s'-s))} + e^{\tau/2}x_0. \]
   Using Lemma~\ref{lem:geomJacobi2}, we find that as long as
   \[ A := R(\tau) e^\tau|s'-s| + e^{\tau/2}|x_0| \]
   is sufficiently small, we can write $M'_\tau$ as the graph of $u'$
   over $N'_\tau$ on $B _{R(\tau)+6}$, with $|u'| < CA$. 

   By our assumptions we have $A < 2e^\tau\gamma$. So if $CR_0 R(\tau)\gamma
   e^\tau < 100^{-1}$ for large $C$, then we will have $|u'| <
   CR_0R(\tau) \gamma e^\tau$. Combining this with the estimate for $\tilde u$ above,
   and increasing $C$, we get the required estimate for $v(x,\tau)$. 
 \end{itemize}
 
  These are the assumptions that we need in order to apply
   Proposition~\ref{prop:rescaledperturb}, if we set $\epsilon =
   C\gamma$. For sufficiently small
   $\delta_0$ we find that the flow $L^{a', \alpha, \Phi'}_t$ does not
   have any cylindrical singularity at $(y,0,0,0)$ with $|y| <
   (C\gamma)^{2/3}$. So if $L^{a', \alpha', \Phi'}_t$ has a
   cylindrical singularity (whether degenerate or not), then
   $|\alpha'-\alpha| \geq (C\gamma)^{2/3} \geq |a-a'|^{3/4}$, if
   $\gamma$ is small enough. 
 \end{proof}
 
 \begin{prop}\label{prop:localperturb2}
   Suppose that $L_t$ has a $\mathcal{C}$-singularity at $(0,0)$, and
   we are in the setting of Proposition~\ref{prop:localperturb}. Then
   for the $\epsilon_0 > 0$ from the Proposition, which depends on the
   flow $S_t$, we have the following. 
   For a dense set of choices of the 
   parameter  $a$ near $0$, the perturbed flows
   $L^a_t$ have no degenerate cylindrical singularity in the parabolic
   $\epsilon_0$-neighborhood of $(0,0)$.
 \end{prop}
 \begin{proof}
   We take $\epsilon_0, \epsilon_1, \delta_0$ to be the constants from
   Proposition~\ref{prop:localperturb}. Let us define
   \[ \mathcal{D} = \{(a, \alpha)\,&:\, |a| < \epsilon_1, |\alpha| <
     \epsilon_0, L^{a,\alpha, \Phi}_t\, \text{ has a degenerate
     } \\
     & \qquad \text{cylindrical singularity at } (0,0) \text{ for some
       } |\Phi| < \epsilon_0\}.\]
   Let us also write $\mathcal{S} = \{a\,:\, (a,\alpha)\in \mathcal{D}
   \text{ for some }\alpha\}$ for the projection of $\mathcal{D}$ onto
   the $a$-component.

   Let $0 < \kappa < \epsilon_1$, and consider a maximal set of pairs $(a_i,
   \alpha_i) \in \mathcal{D}$ such that $|a_i - a_j| > \kappa$ for
   $i\not=j$. Then by Proposition~\ref{prop:localperturb} we have
   $|\alpha_i - \alpha_j| > \kappa^{3/4}$, and so
   there can be at most $\kappa^{-3/4}$ distinct pairs. It follows
   that $\mathcal{S}$ can be covered with $\kappa^{-3/4}$ intervals
   of length $2\kappa$, so the measure of $\mathcal{S}$ satisfies
   $\mu(\mathcal{S})\leq 2\kappa^{1/4}$. Since $\kappa$ can be
   arbitrarily small, it follows that $\mathcal{S}$ has measure zero, and so
   for a dense set of parameters $a$, the flow $L^a_t$ has no
   degenerate cylindrical singularities in the parabolic
   $\epsilon_0$-neighborhood of the origin. 
 \end{proof}

 We can now prove our main results.

 \begin{proof}[Proof of Theorem~\ref{thm:main}]
     We set $R_0, \delta_0$ as in
     Proposition~\ref{prop:localperturb}. Recall that we can assume
     that $S_t$ has only spherical and cylindrical singularities. The
     set $\Sigma^{deg}$ of all degenerate cylindrical singularities of
     $S_t$ is closed, since spherical and nondegenerate cylindrical
     singularities are isolated in spacetime (see \cite{SWX25}). Let us say that a point
     $(x_1,t_1)$ is $\delta$-cylindrical at scale $r$, if there is a
     rotation $Q$ such that
     \[ Q r^{-1} (S_{t_1 - r^2} - x_1) \]
     is $\delta$-graphical over $\mathcal{C}$ on
     $B_{\delta^{-1}}$. From Colding-Minicozzi~\cite{CM15_1} we know
     that for any $\delta > 0$ there is a $\delta' > 0$ such that if
     $(x_1,t_1)\in \Sigma^{deg}$ is $\delta'$-cylindrical at scale $r'$, then it is
     also $\delta$-cylindrical at all scales $r < r'$.
     
     We claim that given any $\delta_1 > 0$, there is a
     small number $r_1 > 0$ such that if $(x_1,t_1)$ is a degenerate
     cylindrical singularity, then it is $\delta_1$-cylindrical at
     all scales $r < r_1$. To see this we can argue by
     contradiction. If it were not true, then we would have a sequence
     $(x_k, t_k)$, and corresponding $r_k\to 0$ such that the $(x_k,
     t_k)$ are not $\delta_1$-cylindrical at scale $r_k$.
     Up to choosing a subsequence the $(x_k, t_k)$
     converge to a (degenerate) cylindrical singularity $(x_\infty,
     t_\infty)$. For a given $\delta_2 > 0$, there is some $r_\infty >
     0$ such that $(x_\infty, t_\infty)$ is $\delta_2$-cylindrical at
     all scales $r < r_\infty$, and from this it follows that
     for sufficiently large $k$, the $(x_k, t_k)$ are also
     $2\delta_2$-cylindrical at scale $r_\infty/2$. If $\delta_2$ is
     chosen sufficiently small (depending on $\delta_1$), then by
     \cite{CM15_1} this implies that for all $r \leq 
     r_\infty/2$ the $(x_k, t_k)$ are $\delta_1$-cylindrical at
     scale $r$, contradicting our assumption.

     Assuming now that all degenerate singularities of $S_t$ are
     $\delta_1$-cylindrical at scales $r < r_1$, it follows
     that if the flow $S'_t$ is sufficiently close to $S_t$,
     then all degenerate singularities of $S'_t$ are
     $\delta_1/2$-cylindrical at scale $r_1/2$. By choosing $\delta_1$
     sufficiently small, depending on $\delta_0$, we can therefore
     assume that the degenerate cylindrical singularities of all flows
     $S'_t$ that are sufficiently close to $S_t$ are
     $\delta_0/2$-cylindrical at scale $r$ for all $r \leq r_1/2$. 

     We also recall the constant $\gamma$ from
     Proposition~\ref{prop:barrier}, which can be chosen uniformly for
     all $T$, for all flows $S'_t$ that are sufficiently close to
     $S_t$. Shrinking $r_1$ we can assume that
     $(r_1/2)^2 < \gamma$. We now argue as follows, inductively perturbing
     away the degenerate singularities. We fix a small $\kappa > 0$,
     and $S^i_t$ will be a perturbation of $S^{0}_t := S_t$ such that the
     initial condition  $S^{i}_0$ is $2^{-i}\kappa$-graphical over
     $S^{i-1}_0$. As long as $\kappa$ is sufficiently small, this will
     imply that all $S^i_t$ are sufficiently close to $S_t$ for the
     estimate on the cylindrical scale above, as well as 
     the barrier construction in Proposition~\ref{prop:barrier}, to
     hold for them.

     Let us suppose that $S^i_t$ is smooth up to a time $T_i > 0$. Let
     us consider the set $\Sigma$ of degenerate cylindrical
     singularities of $S^i_t$ in the interval $[T_i, T_i +
     (r_1/10)^2]$ for the $r_1$ found above. We can cover $\Sigma$ with
     a finite collection of parabolic neighborhoods $Q_{\epsilon_0 r_1/10}(x_k,
     t_k)$, where the $(x_k, t_k)$ for $k=1,\ldots, N$ are $\delta_0/2$-cylindrical
     degenerate singularities at scale $r_1/2$ and
     $\epsilon_0 >0$ is the constant from
     Proposition~\ref{prop:localperturb2}. To perturb
     away these singularities we inductively construct perturbations
     $S^{i,k}_t$, satisfying
     \begin{enumerate}
       \item the initial condition $S^{i,k}_0$ is $2^{-i-1}2^{-k}\kappa$-graphical
     over $S^{i,k-1}_0$, and $S^{i,0}_t = S^i_t$, 
       \item $S^{i,k}_t$ is still smooth up to time $T_i$, and it has
         no degenerate cylindrical singularities outside of
         $Q_{\epsilon_0 r_1/10}(x_j, t_j)$ for $j=k+1, \ldots, N$.
       \item any degenerate cylindrical singularity of $S^{i,k}_t$ is
         still $\delta_0/2$-cylindrical at scale $r_1/2$. 
       \end{enumerate}
      Note that the property (3) is automatic from (1) since all the
      flows we construct are still close enough to $S_t$ for the
      estimate on the cylindrical scale to hold. Assuming that we have
      constructed $S^{i,k}_t$, if there are no degenerate cylindrical
      singularities in $Q_{\epsilon_0 
        r_1/10}(x_{k+1}, t_{k+1})$, we let $S^{i,k+1}_t =
      S^{i,k}_t$. Otherwise let $(x_{k+1}', t_{k+1}')$ be 
      such a singularity. It is $\delta_0/2$-cylindrical at scale
      $r_1/2$, and crucially $t_{k+1} - (r_1/2)^2 < T_i$, since we
      have $t_{k+1} \leq T_i + (r_1/10)^2 + (\epsilon_0 r_1/10)^2$ (and
      $\epsilon_0 < 1$).  Since $S^{i,k}_t$ is smooth for $t\in
      [0,t_{k+1} - (r_1/2)^2]$, Proposition~\ref{prop:localperturb2} implies
      that we can find an arbitrarily small perturbation $S^{i,k+1}_0$
      of $S^{i,k}_0$ such that the flow $S^{i, k+1}_t$ 
      has no degenerate cylindrical singularities in
      $Q_{\epsilon_0 r_1/2}(x_{k+1}', t_{k+1}')$. This implies that
      there are no degenerate cylindrical singularities in
      $Q_{\epsilon_0 r_1/10}(x_{k+1}, t_{k+1})$. Since we can take the
      perturbation $S^{i,k+1}_t$ as small as we like, and
      nondegenerate cylindrical singularities cannot become degenerate under a
      small perturbation, we can arrange that (1), (2), (3) hold for
      $S^{i,k+1}_t$.

      We have that $S^{i,N}_t$ is smooth up to time $T_i$, and has
      no degenerate cylindrical singularities in the time interval $[T_i, T_i +
      (r_1/10)^2]$. We define $S^{i+1}_t = S^{i,N}_t$. If $S^{i+1}_t$
      has some singularity before time $T_i + (r_1/10)^2$, then this
      singularity must be spherical or nondegenerate cylindrical, so
      we have found a perturbation of $S_t$ with only spherical or
      nondegenerate cylindrical singularities at the first singular
      time.

      If instead $S^{i+1}_t$ is smooth up to time $T_{i+1} := T_i + (r_1/10)^2$,
      then we can repeat the previous construction to obtain
      $S^{i+2}_t$. Since the time interval is extended by a fixed
      amount $(r_1/10)^2$ each time, after a finite number of steps
      the $S^i_t$ must have a singularity before time $T_{i+1}$ for
      some $i$, and this singularity must be spherical or
      nondegenerate cylindrical. 
 \end{proof}

 \begin{proof}[Proof of Theorem~\ref{thm:surgery}]
   We argue somewhat similarly to the proof of Theorem~\ref{thm:main},
   the main difference is that we do not try to use initial
   perturbations of the flow, and as a result the flows $S^i_t$ that
   we define inductively will not satisfy the mean curvature flow
   equation at all times.

   Let us define $r_1$ as in the previous proof, and let $T_i = i
   (r_1/10)^2$.  Fix a small $0 < \kappa \ll r_1^2$ which will control the size of our
   perturbation. We define $S^i_t$ inductively, satisfying the
   following:
   \begin{enumerate}
   \item For $j=1, \ldots, i-1$ there are ``surgery times''
     \[ \label{eq:tildeT} \tilde{T}_j \in \left[T_j - \frac{r_1^2}{40}, T_j -
         \frac{r_1^2}{60}\right], \]
     such that $S^i_t$ is a mean curvature flow with only spherical
     and nondegenerate cylindrical singularities on
     \[ [0, T_i] \setminus \{\tilde{T}_1, \ldots, \tilde{T}_{i-1}\}. \]
     \item For some $\kappa_j > 0$ the flow $S^i_t$ is smooth on the
       two intervals $(\tilde{T}_j - \kappa_j, \tilde{T}_j)$ and
       $(\tilde{T}_j, \tilde{T}_j + \kappa_j)$, for $j=1, \ldots, i-1$. Both of these extend
       smoothly to $\tilde{T}_j$, and we denote the two one-sided
       limits by $S^i_{\tilde{T}_j,\pm}$. Then for
       some $A_j > 0$ we have that the second fundamental form of
       $S^i_{\tilde{T}_j, -}$ satisfies $|A| < A_j$, and
       $S^i_{\tilde{T}_j,+}$ is $A_j^{-1}\kappa$-graphical over
       $S^i_{\tilde{T}_j,-}$. 
     \item On $[T_{i-1}, T_{ext}+1]$, $S^i_t$ is a (weak) mean
       curvature flow that is $(1-2^{-i})\kappa$-close to $S_t$ under
       a metrization of the convergence of Brakke flows.  In addition
       $S^i_t$ agrees with $S^{i-1}_t$ for $t\in [0, \tilde{T}_{i-1}-\kappa_{i-1}]$. 
   \end{enumerate}

   Decreasing $r_1$ if necessary, we can assume that $S_t$ is smooth
   until time $T_1$, and define $S^1_t = S_t$. Suppose that we have
   already defined $S^i_t$,
   together with suitable $\tilde{T}_j, \kappa_j$. Since $S^i_t$ has
   only spherical and nondegenerate cylindrical singularities on
   $[T_{i-1}, T_i]$, and these singularities are isolated, we can choose
   $\tilde{T}_i$ satisfying \eqref{eq:tildeT}, and $\kappa_i$ such that $S^i_t$ is smooth on
   $I_i = [\tilde{T}_i - \kappa_i, \tilde{T}_i + \kappa_i]$. 

   We will now use Proposition~\ref{prop:localperturb2} similarly to
   the proof of Theorem~\ref{thm:main} above to construct
   an arbitrarily small perturbation $\tilde{S}^i_{\tilde{T}_i}$ of the time slice
   $S^i_{\tilde{T}_i}$, so that the corresponding flow $\tilde{S}^i_t$
   has no degenerate
   cylindrical singularities in the time interval $[T_i, T_{i+1}]$. We
   then define $S^{i+1}_t$ as follows:
   \begin{itemize}
     \item For $t < \tilde{T}_i$, $S^{i+1}_t$ agrees with
       $S^i_t$.
   \item For $t \geq \tilde{T}_i$ we define $S^{i+1}_t =
     \tilde{S}^i_t$.
   \end{itemize}
   By choosing the perturbation $\tilde{S}^i_{\tilde{T}_i}$
   sufficiently small, this flow $S^{i+1}_t$ will satisfy the
   conditions (1), (2), (3) above.

   It remains to verify that we can construct the required perturbations
   $\tilde{S}^i_{\tilde{T}_i}$ of $S^i_{\tilde{T}_i}$. Suppose
   that $S^i_t$ has a degenerate singularity $(x', t')$, with $t' \in
   [T_i, T_i + (r_1/10)^2]$. For simplicity let us assume for now that all other
   degenerate singularities in the time interval $[T_i, T_i +
   (r_1/10)^2]$ are contained in the neighborhood $Q_{\epsilon_0
     r_1/10}(x', t')$. 

   Let us define $(r')^2 = t' - \tilde{T}_i$.  Note that we have
   \[  (r_1/8)^2 < \frac{r_1^2}{60} \leq (r')^2 = t' - \tilde{T}_i \leq
     \frac{r_1^2}{100} + \frac{r_1^2}{40} \leq (r_1/4)^2. \]
   This means that $(x', t')$ is
   $\delta_0$-cylindrical at scale $r'$. Using this,
   Proposition~\ref{prop:localperturb2} implies that we can find
   arbitrarily small perturbations $\tilde{S}^i_{\tilde{T}_i}$of
   $S^i_{\tilde{T}_i}$, such that the corresponding flow
   $\tilde{S}^i_t$ has no degenerate cylindrical singularities in
   $Q_{\epsilon_0 r'}(x', t')$. By the lower bound above, this
   neighborhood contains $Q_{\epsilon_0 r_1/8}(x', t')$, which by our
   assumption contained all the degenerate singularities of $S^i_t$ on
   $[T_i, T_i + (r_1/10)^2]$. We can then define $S^{i+1}_t$ as
   described above.

   If $S^i_t$ has additional degenerate singularities on $[T_i, T_i +
   (r_1/10)^2]$, then we can argue with a covering, and successive
   perturbations as in the proof of Theorem~\ref{thm:main}.
 \end{proof}
 
\bibliography{mybib}{}
\bibliographystyle{plain}

\end{document}